\documentclass{amsart}

\usepackage{amsthm}
\usepackage{amsmath}
\usepackage{amssymb}
\usepackage{amsfonts} 
\usepackage{amscd}

\theoremstyle{plain}
\newtheorem{theorem}{Theorem}[section]
\newtheorem{lemma}[theorem]{Lemma}
\newtheorem{corollary}[theorem]{Corollary}
\newtheorem{proposition}[theorem]{Proposition}

\theoremstyle{definition}
\newtheorem{definition}[theorem]{Definition}

\theoremstyle{remark}
\newtheorem{remark}{Remark}

\begin{document}

\title[Profinite Conjugacy]{Generalized Bowen-Franks Groups and Profinite Conjugacy of Hyperbolic Toral Automorphisms}

\author[Bakker]{Lennard F. Bakker}
\address{Department of Mathematics, Brigham Young University, Provo, Utah, USA}
\email{bakker@mathematics.byu.edu}
\author[Martins Rodrigues]{Pedro Martins Rodrigues}
\address{Department of Mathematics, Instituto Superior T\'ecnico, Univ. Tec. Lisboa, Lisboa, Portugal}
\email{pmartins@math.ist.utl.pt}

\subjclass[2010]{Primary: 37C15; Secondary: 37C79, 37D20, 20E18} 
\keywords{generalized Bowen-Franks groups, profinite conjugacy, hyperbolic toral automorphism, cyclicity}

\begin{abstract}
We show that a collection of generalized Bowen-Franks group, what we call the principal Bowen-Franks $R$-modules, form a complete set of $R$-module invariants for the equivalence relation of profinite conjugacy for similar hyperbolic toral automorphisms. We also show that these principal generalized Bowen-Franks $R$-modules are the principal invariants in a large class of generalized Bowen-Franks $R$-module invariants for similar hyperbolic automorphisms. We further show how the conjugacy invariant of cyclicity, when applied to hyperbolic toral automorphisms, is realized for profinite conjugacy.
\end{abstract}

\maketitle

\section{Introduction}

Generalized Bowen-Franks groups are abelian group invariants for equivalence relations in symbolic dynamics and toral endomorphisms. The original Bowen-Franks group arose as an invariant of flow equivalence for suspension flows of subshifts of finite type \cite{BF77}. Following the notation of Lind and Marcus \cite{LM21}, the Bowen-Franks group for $A\in M_n({\mathbb Z})$, an $n\times n$ integer matrix representing the subshift of finite type, is
\[ {\rm BF}(A) = {\mathbb Z}^n/{\mathbb Z}^n(I-A),\]
where $I-A$ acts on the right and the elements of ${\mathbb Z}^n$ are row vectors. Franks showed that ${\rm BF}(A)$, together with the flow equivalence invariant ${\rm det}(I-A)$ of Parry and Sullivan \cite{PS75}, are a complete set of computable invariants of shift equivalence for irreducible subshifts of finite type \cite{Fr84}. The Bowen-Franks group ${\rm BF}(A)$ has been an important tool in the classification of reducible subshifts \cite{Hu95} as well as general subshifts \cite{Ma01,Ma21}. More invariants of shift equivalence for subshifts of finite type were introduced by Kitchens that generalized the Bowen-Franks group \cite{Ki98}. These generalized Bowen-Franks groups are of the form
\[ {\rm BF}_g(A) = {\mathbb Z}^n/{\mathbb Z}^n g(A),\]
where $g\in {\mathbb Z}[t]$ is a polynomial with $g(0)$ being a unit in ${\mathbb Z}$, i.e., $g(0) = \pm 1$ \cite{Ki98, LM21}. The original Bowen-Franks group ${\rm BF}(A)$ corresponds to ${\rm BF}_g(A)$ for $g(t) = 1-t$. Another generalization of the Bowen-Franks group ${\rm BF}(A)$ was introduced by one of the authors and Sousa Ramos for toral endomorphisms of ${\mathbb T}^n$ which are represented by $A\in M_n({\mathbb Z})$ \cite{MR05}. These generalized Bowen-Frank groups ${\rm BF}_g(A)$ are invariants of topological conjugacy for $A$ where $g\in {\mathbb Z}[t]$. Topological conjugacy of two toral endomorphisms $A,B\in M_n({\mathbb Z})$ is equivalent to conjugacy by an element of ${\rm GL}_n({\mathbb Z})$, i.e., there exists $C\in {\rm GL}_n({\mathbb Z})$ such that $AC = CB$  (cf \cite{AP65}). For a toral endomorphism $A \in M_n({\mathbb Z})$, each generalized Bowen-Franks group ${\rm BF}_g(A)$ has the additional structure of a right $R$-module for a ring $R$ isomorphic to the matrix ring
\[ {\mathbb Z}[A] = \{ p(A) : p\in {\mathbb Z}[t]\}\]
(see Subsection \ref{sec:Module}). The structure of the finitely generated abelian group ${\rm BF}_g(A)$, with $A$ representing a subshift of finite type or a toral endomorphism, can be determined by computing the Smith normal form of $g(A)$. One of the authors and Sousa Ramos showed, for a given irreducible monic polynomial $f\in {\mathbb Z}[t]$ satisfying $f(0) = \pm 1$ and having square-free discriminant, that all $A\in {\rm GL}_n({\mathbb Z})$ having characteristic polynomial $f(t)$ have $R$-module isomorphic ${\rm BF}_g(A)$ for each $g\in {\mathbb Z}[t]$  \cite{MR05}. Chen showed in some cases that ${\rm BF}_g(A)$, as an abelian group, is determined by the zeta function of $A$ for those $g\in{\mathbb Z}[t]$ with $g(0) = \pm 1$, without assuming irreducibility of the characteristic polynomial of $A$ \cite{Ch09}. When $A\in {\rm GL}_n({\mathbb Z})$ is a hyperbolic toral automorphism (no eigenvalue of $A$ has modulus $1$) and $g_k(t) = t^k-1$, $k\in {\mathbb N}$, each generalized Bowen-Franks group,
\[ {\rm BF}_k(A) = {\mathbb Z}^n/{\mathbb Z}^n g_k(A) = {\mathbb Z}^n/{\mathbb Z}^n(A^k-I),\]
is a finite abelian group that is isomorphic by Pontryagin duality to the set ${\rm Per}_k(A)$ of $k$-periodic points of the left action of $A$ on ${\mathbb T}^n$, where the elements of ${\mathbb T}^n$ are column vectors. Note that the original Bowen-Franks group ${\rm BF}(A)$, corresponding to $g(t) = 1-t$, is isomorphic to ${\rm BF}_1(A)$, corresponding to $g_1(t) = t-1$. The main purpose of this paper is to show, when $A\in {\rm GL}_n({\mathbb Z})$ is hyperbolic, that the collection of {\it principal} generalized Bowen-Franks $R$-modules ${\rm BF}_k(A)$, $k \in {\mathbb N}$, is a complete set of invariants for an equivalence relation called profinite conjugacy.

Profinite conjugacy is a topological $R$-module isomorphism between two profinite $R$-modules. For a hyperbolic $A\in {\rm GL}_n({\mathbb Z})$, the principal Bowen-Franks $R$-modules ${\rm BF}_k(A)$, $k\in {\mathbb N}$, are parameterized by the collection ${\mathcal Q}^{\rm p}$ of polynomials $g_k$, $k\in {\mathbb N}$, and have an inverse limit $G_{ {\mathcal Q}^{\rm p},A}$ that is a profinite $R$-module, whose Pontryagin dual is the direct limit of ${\rm Per}_k(A)$, $k\in {\mathbb N}$, which is the rational torus. For another hyperbolic $B\in{\rm GL}_n({\mathbb Z})$ that is similar to $A$ over ${\mathbb Q}$ i.e., there exists $P\in {\rm GL}_n({\mathbb Q})$ such that $AP = PB$, so it has the same characteristic polynomial as that of $A$, the principal Bowen-Franks groups ${\rm BF}_k(B)$, $k\in {\mathbb N}$, have an inverse limit $R$-module $G_{ {\mathcal Q}^{\rm p},B}$. Similar hyperbolic $A,B\in {\rm GL}_n({\mathbb Z})$ have a common ring $R$ to which both ${\mathbb Z}[A]$ and ${\mathbb Z}[B]$ are isomorphic (see Subsection \ref{sec:Module}), and are said to be strongly ${\rm BF}$-equivalent over ${\mathcal Q}^{\rm p}$ if ${\rm BF}_k(A)$ and ${\rm BF}_k(B)$ are $R$-module isomorphic, written  ${\rm BF}_k(A) \cong_R {\rm BF}_k(B)$, for all $k\in {\mathbb N}$. We showed that if similar hyperbolic $A,B\in {\rm GL}_n({\mathbb Z})$ are strongly ${\rm BF}$-equivalent over ${\mathcal Q}^{\rm p}$, then $G_{ {\mathcal Q}^{\rm p},A}$ and $G_{ {\mathcal Q}^{\rm p},B}$ are topologically $R$-module isomorphic, i.e., there exists a map $\Psi_{ {\mathcal Q}^{\rm p}}:G_{ {\mathcal Q}^{\rm p},A} \to G_{ {\mathcal Q}^{\rm p},B}$ that is an $R$-module isomorphism and a homeomorphism \cite{BM12}. The existence of such a map $\Psi_{ {\mathcal Q}^{\rm p}}:G_{ {\mathcal Q}^{\rm p},A}\to G_{ {\mathcal Q}^{\rm p},B}$ we call a profinite conjugacy over ${\mathcal Q}^{\rm p}$ between similiar hyperbolic $A,B\in {\rm GL}_n({\mathbb Z})$. The main result of this paper includes the converse: for similar hyperbolic $A,B\in{\rm GL}_n({\mathbb Z})$, if $A$ and $B$ are profinitely conjugate over ${\mathcal Q}^{\rm p}$, then $A$ and $B$ are strongly ${\rm BF}$-equivalent over ${\mathcal Q}^{\rm p}$. This establishes the equivalence of strong ${\rm BF}$-equivalence over ${\mathcal Q}^{\rm p}$ and profinite conjugacy over ${\mathcal Q}^{\rm p}$. 

One of the authors and Sousa Ramos proved a result in \cite{MR05} that implies for similar hyperbolic automorphisms $A,B\in {\rm GL}_n({\mathbb Z})$ with irreducible characteristic polynomial that if $A$ and $B$ are strongly {\rm BF}-equivalent over ${\mathcal Q}^{\rm p}$, then $A$ and $B$ are strongly {\rm BF}-equivalent over the collection ${\mathcal Q}^{\rm a}$ of polynomials $g\in {\mathbb Z}[t]$ satisfying ${\rm det}(g(A))\ne 0$, i.e., ${\rm BF}_g(A) \cong_R {\rm BF}_g(B)$ for all $g\in {\mathcal Q}^a$. We show, as a corollary of the main result, that strong {\rm BF}-equivalence of $A$ and $B$ over ${\mathcal Q}^{\rm p}$ implies strong {\rm BF}-equivalence of $A$ and $B$ over ${\mathcal Q}^{\rm a}$ without assuming irreducibility of the characteristic polynomial of the similar hyperbolic $A$ and $B$. Hence, in the irreducible and reducible cases for similar hyperbolic toral automorphisms, the principal Bowen-Franks $R$-modules are the principal invariants in the class of generalized Bowen-Franks $R$-module invariants parameterized by ${\mathcal Q}^{\rm a}$.

For irreducible similar hyperbolic toral automorphisms, the principal generalized Bowen-Franks $R$-modules are a complete set of invariants for the equivalence relations of weak equivalence of the associated ideals and block conjugacy of the matrices as well as for profinite conjugacy. Here irreducible means the characteristic polynomial $f(t)$ for the similar $A,B\in {\rm GL}_n({\mathbb Z})$ is irreducible in ${\mathbb Z}[t]$. Fixing a root $\beta$ of $f(t)$, the selection of a right eigenvector $u = (u_1,\dots,u_n)^{\rm T}$ of $A$, i.e., $Au = \beta u$, with entries $u_i$ in the algebraic number field ${\mathbb Q}(\beta)$ determines a (fractional) ideal 
\[ I_A = u_1{\mathbb Z} + \cdots + u_n {\mathbb Z}\]
in the ring ${\mathbb Z}[\beta]$. By the Lattimer-MacDuffey-Taussky one-to-one correspondence \cite{LM33,Ta78}, different choices of right eigenvectors determine arithmetically equivalent ideals in ${\mathbb Z}[\beta]$. Similarly the selection of a right eigenvector $v = (v_1,\dots,v_n)^{\rm T}$ for $B$ determines a (fractional) ideal $I_B$ in ${\mathbb Z}[\beta]$. Two (fractional) ideals $I$ and $J$ of ${\mathbb Z}[\beta]$ are weakly equivalent if there exists (fractional) ideals $X$ and $Y$ of ${\mathbb Z}[\beta]$ such that $IX = J$ and $JY = I$ \cite{DTZ62}. We showed that for similar hyperbolic $A,B\in {\rm GL}_n({\mathbb Z})$ with an irreducible characteristic polynomial, if $I_A$ and $I_B$ are weakly equivalent, then $A$ and $B$ are strongly {\rm BF}-equivalent over ${\mathcal Q}^{\rm a}$ \cite{BM12}. Since $A$ is hyperbolic, each $g_k(A) = A^k-I$ is invertible, and so $I_A$ and $I_B$ being weakly equivalent implies that $A$ and $B$ are strongly ${\rm BF}$-equivalent over ${\mathcal Q}^{\rm p}$. We showed that for similar hyperbolic $A,B\in {\rm GL}_n({\mathbb Z})$ with an irreducible characteristic polynomial, if $A$ and $B$ are profinitely conjugate over ${\mathcal Q}^{\rm p}$, then $I_A$ and $I_B$ are weakly equivalent \cite{BMM16}. Similar $A,B\in {\rm GL}_n({\mathbb Z})$ (not assumed hyperbolic) with an irreducible characteristic polynomial are called block conjugate if there exists $A^\prime,B^\prime \in {\rm GL}_n({\mathbb Z})$ and $M,N\in {\rm GL}_{2n}({\mathbb Z})$ such that the block diagonal matrices $A\oplus A^\prime$, $B\oplus B$, $A\oplus A$, and $B\oplus B^\prime$ satisfy
\[ (A\oplus A^\prime)M = M(B\oplus B) {\rm\ and \ }(A\oplus A)N = N(B\oplus B^\prime).\]
We showed that similar $A,B\in {\rm GL}_n({\mathbb Z})$ with irreducible characteristic polynomial are block conjugate if and only if the associated ideals $I_A$ and $I_B$ are weakly equivalent \cite{BM19}. Applying this to similar hyperbolic $A,B\in {\rm GL}_n({\mathbb Z})$ with irreducible characteristic polynomial gives $I_A$ and $I_B$ are weakly equivalent if and only if $A$ and $B$ are block conjugate. Putting these results together gives that the following are equivalent for irreducible similar hyperbolic $A,B\in {\rm GL}_n({\mathbb Z})$:
\begin{enumerate}
\item[(a)] $I_A$ and $I_B$ are weakly equivalent,
\item[(b)] $A$ and $B$ are block conjugate
\item[(c)] $A$ and $B$ are profinitely conjugate over ${\mathcal Q}^{\rm p}$,
\item[(d)] $A$ and $B$ are strongly ${\rm BF}$-equivalent over ${\mathcal Q}^{\rm p}$.
\end{enumerate}
This establishes the principal generalized Bowen-Franks $R$-modules as a set of complete invariants for the equivalence relations listed in items (a), (b), and (c) when $A$ and $B$ are irreducible similar hyperbolic toral automorphisms.

The structure of this paper is as follows. Section 2 states some basic definitions and theory for profinite groups, as found in \cite{RZ10}, that will be used in the proof of the main result. In particular profinite topologies, profinite completions, and cofinality are detailed. Section 3 applies the results from Section 2 to profinite completions of ${\mathbb Z}^n$ with particular attention to the full profinite topology of ${\mathbb Z}^n$ and its accompanying completion. In Section 4 we prove for ${\mathcal Q}$ being either ${\mathcal Q}^p$ or ${\mathcal Q}^a$ the main result that profinite conjugacy over ${\mathcal Q}$ implies strong ${\rm BF}$-equivalence over ${\mathcal Q}$. Prior to giving the proof in this section, the construction of the profinite $R$-modules $G_{{\mathcal Q},A}$ for ${\mathcal Q} = {\mathcal Q}^{\rm p}$ and ${\mathcal Q}^{\rm a}$ is given. We already have a proof of the main result over ${\mathcal Q}^p$ for {\it irreducible} similar hyperbolic toral automorphisms, as this follows from the implications (c) $\Rightarrow$ (a) $\Rightarrow$ (d) stated in the previous paragraph. The proof presented in Section 4 does not assume irreducibility and does not use weak equivalence. In Section 5 we present some properties of profinite conjugacies for similar hyperbolic toral automorphisms. We correct a characterization from \cite{BM12} of when profinitely conjugate similar hyperbolic toral automorphisms are conjugate. We also show how profinite conjugacy realizes the conjugacy invariant of cyclicity when applied to hyperbolic toral automorphisms.

\section{Profinite Theory}

\subsection{Profinite Topologies}

A profinite topology on an abelian group $G$ is a translation invariant topology determined by selecting a neighbourhood basis of the identity element $0$ of $G$ in the following way. Let ${\mathcal N}$ be a nonempty collection of finite index subgroups of $G$ that is filtered from below, i.e., for any $N_1,N_2\in{\mathcal N}$ there exists $N_3\in{\mathcal N}$ such that $N_3\subset N_1\cap N_2$. The filtered from below property makes ${\mathcal N}$ into a directed poset for the binary relation $\preceq$ on ${\mathcal N}$ defined by $N_1\preceq N_2$ if and only if $N_2 \subset N_1$. The ``directed'' part is that for any $N_1,N_2\in{\mathcal N}$ there exists $N_3\in{\mathcal N}$ such that $N_1\preceq N_3$ and $N_2\preceq N_3$. Viewing ${\mathcal N}$ as a fundamental system of neighbourhoods (which is the same as a neighbourhood basis) of the identity element $0$ of $G$ determines a profinite topology on $G$ (see \cite{Bo89}).

For a nonempty collection ${\mathcal N}$ of filtered from below finite index subgroups of $G$, the type of profinite topology on $G$ determined by ${\mathcal N}$ is described by the class ${\mathcal C}$ of finite abelian groups to which all the quotients $G/N$, $N\in {\mathcal N}$, belong. A specified class ${\mathcal C}$ of finite abelian groups always contains all isomorphic copies of the groups in it.  The simplest structure on a class ${\mathcal C}$ is that of a formation. A class ${\mathcal C}$ of finite abelian groups is called a formation if it satisfies the following two properties: (1) if $K\in {\mathcal C}$ and $H\leq K$, then $K/H\in {\mathcal C}$, and (2) if $H_j\in {\mathcal C}$ for $j=1,\dots,r$ and $K$ is a subgroup of the direct product $H_1\times\cdots\times H_r$ such that for the canonical projections (epimorphisms)
\[ q_j: H_1\times\cdots \times H_j\times\cdots\times H_r\to H_j,\]
the restrictions $q_j\vert_K:K\to H_j$ are surjections for all $j=1,\dots,r$, then $K\in {\mathcal C}$. The first property is the closure of ${\mathcal C}$ under taking quotients. The second property is the closure of ${\mathcal C}$ under taking subdirect products. The second property implies the closure of ${\mathcal C}$ under taking direct products. The class ${\mathcal C}$ of all finite abelian groups is readily verified as a formation.

For a formation ${\mathcal C}$ of finite abelian groups, a pro-${\mathcal C}$ topology on $G$ is a profinite topology determined by a nonempty collection ${\mathcal N}$ of filtered from below finite indexed subgroups of $G$ that satisfies $G/N\in {\mathcal C}$ for all $N\in{\mathcal N}$. The full pro-${\mathcal C}$ topology on $G$ is the pro-${\mathcal C}$ topology determined by the nonempty filtered from below collection
\[ {\mathcal N}_{\mathcal C}(G) = \{ N \leq_f G : G/N \in {\mathcal C}\},\]
where $\leq_f$ means finite index subgroup.

\subsection{Profinite Completions}

A profinite completion of $G$ is induced by a profinite topology on $G$. Let ${\mathcal N}$ be a nonempty collection of finite index subgroups of $G$ that is filtered from below. For each $N\in{\mathcal N}$ equip the finite abelian quotient group $G/N$ with the discrete topology. Each $G/N$ is a compact, Hausdorff, totally disconnected topological group. For $N,M\in{\mathcal N}$ satisfying $M\preceq N$ and $M\ne N$, the natural epimorphism
\[ \varphi_{NM}:G/N \to G/M {\rm \ given\ by\ } x + N \mapsto x +M\]
is continuous. Denote the coset $x+N$ of $G/N$ by $[x]_N$. For each $N\in {\mathcal N}$ we set $\varphi_{NN}$ to be the identity map from $G/N$ to $G/N$. For all $N,M,P\in {\mathcal N}$ satisfying $N\preceq M\preceq P$ there holds $\varphi_{PN} = \varphi_{MN}\varphi_{PM}$. The continuous epimorphisms $\varphi_{NM}$ for $M\preceq N$ are called transitions maps. The collection $\{ G/N, \varphi_{NM}, {\mathcal N}\}$ is a surjective inverse system of topological finite abelian groups. The profinite completion of $G$ with respect to the profinite topology on $G$ determined by ${\mathcal N}$ is the inverse limit of the surjective inverse system,
\[ {\mathcal K}_{\mathcal N}(G) = \varprojlim_{N\in{\mathcal N}} G/N,\]
whose elements are the coherent tuples $\{[x_N]_N\}_{N\in{\mathcal N}}$ in $\prod_{N\in{\mathcal N}} G/N$, i.e., those tuples satisfying
\[ [x_M]_M = \varphi_{NM}([x_N]_N) = [x_N]_M\]
whenever $M\preceq N$. The profinite completion ${\mathcal K}_{\mathcal N}(G)$ is a nonempty, compact, Hausdorff, totally disconnected topological abelian group. The profinite completion ${\mathcal K}_{\mathcal N}(G)$ is called a profinite abelian group.

Several natural continuous homomorphisms are associated to the profinite abelian group ${\mathcal K}_{\mathcal N}(G)$. There are the continuous epimorphisms
\[ \varphi_{N}:{\mathcal K}_{\mathcal N}(G) \to G/N, \ \ \varphi_{N}( [x_M]_{M\in{\mathcal N}}) = [x_N],\ N\in {\mathcal N}.\]
The continuous epimorphisms $\iota_N:G\to G/N$, $N\in {\mathcal N}$, induce the continuous homomorphism $\iota_{\mathcal N}:G \to {\mathcal K}_{\mathcal N}(G)$ given by
\[ \iota_{\mathcal N}(x) =  \{[x]_N\}_{N\in{\mathcal N}},\]
which take each $x\in G$ to the constant tuple $\{[x]_N\}_{N\in{\mathcal N}}$ in ${\mathcal K}_{\mathcal N}(G)$. The map $\iota_{\mathcal N}$ has the property that $\iota_{\mathcal N}(G)$ is dense in ${\mathcal K}_{\mathcal N}(G)$. Furthermore $\iota_{\mathcal N}$ is a monomorphism if and only if 
\[ \bigcap_{N\in {\mathcal N}} N = \{0\}.\]

For a fixed formation ${\mathcal C}$ of finite abelian groups, a pro-${\mathcal C}$ completion of $G$ is the profinite abelian group ${\mathcal K}_{\mathcal N}(G)$ when $G/N\in {\mathcal C}$ for all $N\in {\mathcal N}$. The profinite completion of $G$ induced by the full profinite topology determined by ${\mathcal N}_{\mathcal C}(G)$ will be denoted by ${\mathcal K}_{\mathcal C}(G)$ instead of the more cumbersome ${\mathcal K}_{ {\mathcal N}_{\mathcal C}(G)}(G)$. The continuous homomorphism $\iota_{ {\mathcal N}_{\mathcal C}}(G):G \to {\mathcal K}_{\mathcal C}(G)$ will be denoted by the simpler $\iota_{\mathcal C}(G)$ as well. The profinite abelian group ${\mathcal K}_{\mathcal C}(G)$ is called residually ${\mathcal C}$ if
\[ \bigcap_{N\in {\mathcal N}_{\mathcal C}(G)} N = \{0\}.\]
This condition is equivalent to the continuous homomorphism $\iota_{\mathcal C}$ being a monomorphism.

We need a few additional notations to state the following result that relates the full pro-${\mathcal C}$ topology on $G$ with the topology on the profinite completion ${\mathcal K}_{\mathcal C}(G)$. The notation $\{ U: U\leq_o G;{\mathcal N}\}$ means the collection of all open subgroups $U$ of $G$ with respect to the profinite topology determined by a nonempty filtered from below collection ${\mathcal N}$ of finite index subgroups of $G$. For groups $H$ and $K$ with $H\leq K$ the notation $[K:H]$ means the index of $H$ in $K$. The notation $H\leq_o G$ means $H$ is an open subgroup of a topological group $G$. The following result (Proposition 3.2.2 in \cite{RZ10}) is presented in the abelian category, and will be used to prove the main result.

\begin{theorem}\label{fullproCtopology} Let ${\mathcal C}$ be a formation of finite abelian groups and suppose $G$ is a residually ${\mathcal C}$ group. Identify $G$ with its embedded image $\iota_{\mathcal C}(G)$ in its pro-${\mathcal C}$ completion ${\mathcal K}_{\mathcal C}(G)$. Let $\overline{X}$ denote the closure in ${\mathcal K}_{\mathcal C}(G)$ of a subset of $X$ in $G$.
\begin{enumerate}
\item[(a)] Let
\[ \Upsilon: \{ U: U \leq_o G; {\mathcal N}_{\mathcal C}(G)\} \to \{ V: V \leq_o {\mathcal K}_{\mathcal C}(G) \}\]
be the mapping that assigns to each open subgroup $H$ in $G$ $($with respect to the full profinite topology$)$ the closure $\Upsilon(H)=\overline{H}$ in ${\mathcal K}_{\mathcal C}(G)$. Then $\Upsilon$ is a one-to-one correspondence between the set of all open subgroups in the full pro-${\mathcal C}$ topology of $G$ and the set of all open subgroups of ${\mathcal K}_{\mathcal C}(G)$.
The inverse of $\Upsilon$ is
\[ V\to V\cap\iota_{\mathcal C}(G),\]
where in particular $\overline{ V\cap \iota_{\mathcal C}(G)} = V$ when
$V \leq_o {\mathcal K}_{\mathcal C}(G)$. 
\item[(b)] The topology of ${\mathcal K}_{\mathcal C}(G)$ induces on $G$ its full pro-${\mathcal C}$ topology via the map $V \to V\cap \iota_{\mathcal C}(G)$ and the identification of $G$ with $\iota_{\mathcal C}(G)$.
\item[(c)] If $H,K \in \{ U : U \leq_o G; {\mathcal N}_{\mathcal C}(G)\}$ and $H \leq K$, then $K/H \cong \overline{K}/\overline{H}$ which implies $[K:H] = [\overline {K}:\overline{H}]$.
\end{enumerate}
\end{theorem}

\subsection{Cofinality}

Different choices of fundamental systems of neighbourhoods of the identity element may lead to topologically isomorphic profinite completions, where by topological isomorphism is meant an isomorphism that is also a homeomorphism. This can happen because of cofinality. Suppose ${\mathcal N}$ and ${\mathcal M}$ are nonempty filtered from below collections of finite index subgroups of $G$. The collection ${\mathcal M}$ is cofinal in ${\mathcal N}$ if ${\mathcal M} \subset {\mathcal N}$ and for every $N\in {\mathcal N}$ there exists $M\in {\mathcal M}$ such that $N \preceq M$, i.e., $M\subset N$. The following result (Lemma 1.1.9 in \cite{RZ10}) is presented in the notation of Subsection 2.2.

\begin{proposition}\label{topisoiota} For nonempty collections ${\mathcal N}$ and ${\mathcal M}$ of finite index subgroups of $G$ filtered from below, if ${\mathcal M}$ is cofinal in ${\mathcal N}$, then there exists a topological isomorphism $\Theta_{\mathcal{NM}}: {\mathcal K}_{\mathcal N}(G) \to {\mathcal K}_{\mathcal M}(G)$, given explicitly by
\[ \Theta_{ \mathcal{NM}} ( \{ [x_N]_N\}_{N\in{\mathcal N}}) = \{ [x_N]_N\}_{N\in{\mathcal M}},\]
that satisfies $\Theta_{\mathcal {NM}}\iota_{\mathcal N} = \iota_{\mathcal M}$.
\end{proposition}

The topological isomorphism $\Theta_{ {\mathcal N}{\mathcal M}}$ in Proposition \ref{topisoiota} is the inverse of the map constructed in Lemma 1.1.9 in \cite{RZ10}; the map $\Theta_{{\mathcal N}{\mathcal M}}$ ``forgets'' those $[x_N]_N$ in the coherent tuple $\{ [x_N]_N\}_ {N\in{ {\mathcal N}}}$ for $N\in {\mathcal N}\setminus{\mathcal M}$. A direct consequence of Proposition \ref{topisoiota} is that the injectivity of the homomorphisms $\iota_N$ and $\iota_M$ are related when ${\mathcal M}$ is cofinal in ${\mathcal N}$.

\begin{corollary}\label{iotainjective} For nonempty collections ${\mathcal N}$ and ${\mathcal M}$ of finite index subgroups of $G$ filtered from below, suppose ${\mathcal M}$ is cofinal in ${\mathcal N}$. Then $\iota_{\mathcal M}$ is a monomorphism if and only if $\iota_{\mathcal N}$ is a monomorphism.
\end{corollary}

A pro-${\mathcal C}$ topology on $G$ determined by ${\mathcal N}$ is the full pro-${\mathcal C}$ topology whenever ${\mathcal N}$ is cofinal in ${\mathcal N}_{\mathcal C}(G)$. This follows from Proposition \ref{topisoiota} and part (b) of Theorem \ref{fullproCtopology}. We present a proof of this that doesn't pass through the profinite completions.

\begin{proposition}\label{fullprofinite} Let ${\mathcal C}$ be a formation of finite abelian groups. If ${\mathcal N}$ is cofinal in ${\mathcal N}_{\mathcal C}(G)$, then the pro-${\mathcal C}$ topology on $G$ determined by ${\mathcal N}$ is the full pro-${\mathcal C}$ topology on $G$. 
\end{proposition}

\begin{proof} By translation invariance it suffices to show for a subset $U$ containing $0$ that $U$ is open with respect to the pro-${\mathcal C}$ topology determined by ${\mathcal N}$ if and only if $U$ is open with respect to the full pro-${\mathcal C}$ topology.

Suppose $U$ is an open subset of $G$ containing $0$ with respect to the pro-${\mathcal C}$ topology determined by ${\mathcal N}$. For each $x\in U$ there exists $N_x\in {\mathcal N}$ such that $x\in x+N_x\subset U$, where $x+N_x$ is the translation of $N_x$ by $x$. Since ${\mathcal N}$ is cofinal in ${\mathcal N}_{\mathcal C}(G)$, we have ${\mathcal N}\subset {\mathcal N}_{\mathcal C}(G)$, so that $N_x\in {\mathcal N}_{\mathcal C}(G)$ for all $x\in U$. Thus for every $x\in U$ there exists $N_x\in {\mathcal N}_{\mathcal C}(G)$ such that $x\in x+N_x\subset U$, and thus $U$ is open in the full pro-${\mathcal C}$ topology.

Now suppose that $U$ is an open subset of $G$ containing $0$ with respect to the full pro-${\mathcal C}$ topology. For each $x\in U$ there exist $F_x\in {\mathcal N}_{\mathcal C}(G)$ such that $x\in x+F_x\subset U$. Since ${\mathcal N}$ is cofinal in ${\mathcal N}_{\mathcal C}(G)$ there exist $N_x\in {\mathcal N}$ such that $0\in N_x\subset F_x$. This implies that $x\in x+N_x\subset x+F_x\subset U$. Thus for every $x\in U$ there exist $N_x\in {\mathcal N}$ such that $x\in x+N_x\subset U$, and thus $U$ is open in the pro-${\mathcal C}$ topology determined by ${\mathcal N}$.
\end{proof}

\begin{remark}\label{applyfullproC} The hypotheses for Theorem \ref{fullproCtopology} are that $G$ is residually ${\mathcal C}$, i.e., the map $\iota_{\mathcal C}:G\to {\mathcal K}_{\mathcal C}(G)$ is a monomorphism, and that the pro-${\mathcal C}$ topology on $G$ is the full pro-${\mathcal C}$ topology, i.e., determined by ${\mathcal N}_{\mathcal C}(G)$. For a pro-${\mathcal C}$ completion ${\mathcal K}_{\mathcal N}(G)$, if ${\mathcal N}$ is cofinal in ${\mathcal N}_{\mathcal C}(G)$, then it follows from Proposition \ref{fullprofinite} that the pro-${\mathcal C}$ topology determined on $G$ by ${\mathcal N}$ is the full pro-${\mathcal C}$ topology, and it follows from Corollary \ref{iotainjective} that $\iota_{\mathcal N}$ is a monomorphism, i.e., $G$ is residually ${\mathcal C}$, hence we can apply Theorem \ref{fullproCtopology} with ${\mathcal K}_{\mathcal C}(G)$ replaced by ${\mathcal K}_{\mathcal N}(G)$, ${\mathcal N}_{\mathcal C}(G)$ replaced with ${\mathcal N}$, and $\iota_{\mathcal C}$ replaced with $\iota_{\mathcal N}$.
\end{remark}

\section{The Profinite Completion of ${\mathbb Z}^n$}

\subsection{The Abelian Group ${\mathbb Z}^n$} We set $G$ to be free abelian group ${\mathbb Z}^n$ of rank $n\geq 2$. We represent the elements of ${\mathbb Z}^n$ as row vectors in terms of the standard basis $\{e_1,e_2,\dots,e_n\}$ where $e_1 = (1,0,\dots,0)$, $e_2=(0,1,\dots,0)$, and $e_n = (0,0,\dots,1)$, so that each $x\in {\mathbb Z}^n$ is expressed as
\[ x = x_1 e_1 + x_2 e_2 + \cdots + x_n e_n\]
for $x_1,x_2,\dots,x_n\in{\mathbb Z}$. 

\subsection{The Full Pro-${\mathcal C}$ Topology} For ${\mathcal C}$ fixed as the formation of all finite abelian groups, the full pro-${\mathcal C}$ topology on ${\mathbb Z}^n$ is determined by the nonempty collection of all finite index subgroups of ${\mathbb Z}^n$,
\[ {\mathcal N}_{\mathcal C}({\mathbb Z}^n) = \{N \leq_f {\mathbb Z}^n : {\mathbb Z}^n/N \in {\mathcal C}\} = \{N:N\leq_f {\mathbb Z}^n\}.\]
To simplify notation we denote ${\mathcal N}_{\mathcal C}({\mathbb Z}^n)$ by ${\mathcal F}$. The profinite completion of ${\mathbb Z}^n$ induced by ${\mathcal F}$ will be denoted by ${\mathcal K}_{\mathcal F}({\mathbb Z}^n)$ and the continuous homomorphism ${\mathbb Z}^n \to  {\mathcal K}_{\mathcal F}({\mathbb Z}^n)$ will be denoted by $\iota_{\mathcal F}$. 

To show that ${\mathbb Z}^n$ is residually ${\mathcal C}$ we use the collection
\[ {\mathcal D} = \{ d{\mathbb Z}^n: d\in {\mathbb N}\}\]
and Corollary \ref{iotainjective}. For each $d\in {\mathbb N}$ the subgroup $d{\mathbb Z}^n$ has the finite index $[{\mathbb Z}^n:d{\mathbb Z}^n] = d^n$. The collection ${\mathcal D}$ is filtered from below because for $d_1,d_2\in {\mathbb Z}$ the choice of $d_3 = {\rm lcm}(d_1,d_2)$ satisfies
\[ d_3{\mathbb Z}^n \subset d_1{\mathbb Z}^n \cap d_2{\mathbb Z}^n,\]
i.e., for $(d_3m_1,d_3m_2,\dots,d_3m_n)\in d_3{\mathbb Z}^n$ the relations $d_1\mid d_3$ and $d_2\mid d_3$ imply that $d_1\mid d_3 m_i$ and $d_2\mid d_3 m_i$ for all $i=1,2,\dots,n$. In proving that ${\mathcal D}$ is cofinal in ${\mathcal F}$ we make use of the result that states for a free abelian group $F$ of finite rank $n$ and a subgroup $H$ of rank $n$ there exists a basis $\{x_1,x_2,\dots,x_n\}$ of $F$ and positive integers $d_1,d_2,\dots,d_n$ satisfying $d_1\mid d_2\mid \cdots \mid d_n$ such that $\{ d_1x_1,d_2x_2,\dots,d_nx_n\}$ is a basis for the subgroup $H$ of $F$ (see Theorem 1.6 in \cite{Hu74}).

\begin{lemma}\label{kproduct} For each $n\geq 2$ the collection ${\mathcal D}$ is cofinal in ${\mathcal F}$.
\end{lemma}

\begin{proof} Clearly ${\mathcal D}\subset {\mathcal F}$. Let $N\in {\mathcal F}$. Since $N\leq_f {\mathbb Z}^n$, there exists a basis $\xi_1,\xi_2,\dots,\xi_n$ of ${\mathbb Z}^n$ and positive integers $d_1,d_2,\dots,d_n$ such that $d_1 \mid d_2 \mid \cdots \mid d_n$ and 
\[d_1\xi_1, d_2\xi_2,\dots, d_n\xi_n\]
is a basis of $N$ . Since $d_1\mid d_2\mid \cdots\mid d_n$, all of the elements $d_n\xi_1,d_n\xi_2,\dots,d_n\xi_n$ belong to $N$. Set $d=d_n$. Since $\xi_1,\xi_2,\dots,\xi_n$ is a basis for ${\mathbb Z}^n$, we have
\[ d(m_1\xi_1+m_2\xi_2+\cdots+m_n\xi_n)\in N\]
for all $m_1,m_2,\dots,m_n\in{\mathbb Z}$. This implies that $d{\mathbb Z}^n \subset N$ where $d{\mathbb Z}^n \in {\mathcal D}$. Thus for every $N\in {\mathcal F}$ there exists $d{\mathbb Z}^n\in {\mathcal D}$ such that $N\preceq d{\mathbb Z}^n$, making ${\mathcal D}$ cofinal in ${\mathcal F}$.
\end{proof}
 
By Lemma \ref{kproduct} and Proposition \ref{fullprofinite}, the cofinal collection ${\mathcal D}$ induces the full pro-${\mathcal C}$ topology on ${\mathbb Z}^n$. Furthermore, by Proposition \ref{topisoiota} the pro-${\mathcal C}$ completion ${\mathcal K}_{\mathcal D}({\mathbb Z}^n)$ is topologically isomorphic to ${\mathcal K}_{\mathcal C}({\mathbb Z}^n)$.

\begin{lemma}\label{residualD} For each $n\geq 2$, the map $\iota_{{\mathcal D}}:{\mathbb Z}^n \to {\mathcal K}_{ {\mathcal D}}({\mathbb Z}^n)$ is a monomorphism.
\end{lemma}

\begin{proof} Suppose
\[ m \in \bigcap_{d\in{\mathbb N}} d{\mathbb Z}^n.\]
Then $m = (m_1,\dots,m_n) \in d{\mathbb Z}^n$ for all $d\in{\mathbb N}$. In particular each $m_i$ is divisible by every $d\in{\mathbb N}$. If $m_i\ne 0$, there is $d > \vert m_i\vert + 1$ for which $m_i$ is not divisible by $d$. This means $m_i=0$, implying that $m=0$.
\end{proof}

\begin{proposition}\label{residual} For each $n\geq 2$, the map $\iota_{{\mathcal F}}:{\mathbb Z}^n \to {\mathcal K}_{{\mathcal F}}({\mathbb Z}^n)$ is a monomorphism.
\end{proposition}

\begin{proof} Since ${\mathcal D}$ is cofinal in ${\mathcal F}$ by Lemma \ref{kproduct} and $\iota_{\mathcal D}$ is a monomorphism by Lemma \ref{residualD}, it follows by Corollary \ref{iotainjective} that $\iota_{\mathcal F}$ is a monomorphism.
\end{proof}

\begin{corollary}\label{Znresfin} The pro-${\mathcal C}$ group ${\mathcal K}_{\mathcal F}({\mathbb Z}^n)$ is residually ${\mathcal C}$.
\end{corollary}

\section{Strong BF-Equivalence}

\subsection{Dynamically Defined Pro-${\mathcal C}$ Topologies}

For similar hyperbolic $A,B\in{\rm GL}_n({\mathbb Z})$, all of the generalized Bowen-Franks $R$-modules, viewed as groups, naturally correspond to a pair of pro-${\mathcal C}$ topologies on ${\mathbb Z}^n$. This pair of pro-${\mathcal C}$ topologies is parameterized by the collection of polynomials
\[ {\mathcal Q}^{\rm a} = \{ g\in {\mathbb Z}[t] : {\rm det}(g(A))\ne 0\}.\]
The invertibility of $g(B)$ follows from the invertibility of $g(A)$ and the similarity of $A$ with $B$, and so the invertibility of $g(B)$ is not explicitly stated in the definition of ${\mathcal Q}^{\rm a}$. Set $N_{g,A} = {\mathbb Z}^n g(A)$, $N_{g,B} = {\mathbb Z}^n g(B)$, ${\mathcal N}_{ {\mathcal Q}^{\rm a},A} = \{ N_{g,A}: g\in {\mathcal Q}^{\rm a} \}$, and ${\mathcal N}_{ {\mathcal Q}^{\rm a},B} = \{ N_{g,B}: g\in {\mathcal Q}^{\rm a}\}$.

\begin{lemma}\label{Qall} For similar hyperbolic $A,B\in {\rm GL}_n({\mathbb Z})$, the collections ${\mathcal N}_{ {\mathcal Q}^{\rm a},A}$ and ${\mathcal N}_{ {\mathcal Q}^{\rm a},B}$ are pro-${\mathcal C}$ topologies on ${\mathbb Z}^n$.
\end{lemma}

\begin{proof} For each $g\in {\mathcal Q}^{\rm a}$, the invertibility of $g(A),g(B)\in M_n({\mathbb Z})$ implies by their Smith normal form that $N_{g,A}$ and $N_{g,B}$ are finite index subgroups of ${\mathbb Z}^n$. The filtered from below property follows from ${\mathcal Q}^{\rm a}$ being a multiplicative set, i.e., closed under multiplication. For $g_1,g_2\in {\mathcal Q}^{\rm a}$ there holds $g_1,g_2\in {\mathbb Z}[t]$ with $g_1(A)$ and $g_2(A)$ invertible, so that setting $g_3 = g_1g_2$ there holds $g_3\in {\mathbb Z}[t]$ with $g_3(A)$ invertible, hence $g_3\in {\mathcal Q}^{\rm a}$. If $g_1,g_2\in {\mathcal Q}^{\rm a}$ then for $g_3 = g_1g_2\in {\mathcal Q}^{\rm a}$ and $m\in {\mathbb Z}^n g_3(A)$ there is $z\in {\mathbb Z}^n$ such that $m = z g_3(A)$, hence there holds $m = (zg_2(A))g_1(A)$ and $m = (zg_1(A))g_2(A)$, implying that $N_{g_3,A} \subset N_{g_1,A}\cap N_{g_2,A}$. A similar argument gives $N_{g_3,B}\subset N_{g_1,B} \cap N_{g_2,B}$.
\end{proof}

For similar hyperbolic $A,B\in {\rm GL}_n({\mathbb Z})$, the principal generalized Bowen-Franks $R$-modules, viewed as groups, naturally correspond to another pair of pro-${\mathcal C}$ topologies on ${\mathbb Z}^n$. This pair of pro-${\mathcal C}$ topologies is parameterized by the collection of polynomials
\[ {\mathcal Q}^{\rm p} = \{ g_k(t) = t^k-1: k\in {\mathbb N}\}.\]
To simplify notation, set $N_{k,A} = N_{g_k,A}$ and $N_{k,B} = N_{g_k,B}$. Then ${\mathcal N}_{ {\mathcal Q}^{\rm p},A} = \{ N_{k,A}: k\in {\mathbb N}\}$ and ${\mathcal N}_{ {\mathcal Q}^{\rm p},B} = \{ N_{k,B} : k\in {\mathbb N}\}$.

\begin{lemma}\label{Qprin} For similar hyperbolic $A,B\in {\rm GL}_n({\mathbb Z})$, there holds ${\mathcal Q}^{\rm p} \subset {\mathcal Q}^{\rm a}$, and the collections ${\mathcal N}_{ {\mathcal Q}^{\rm p},A}$ and ${\mathcal N}_{ {\mathcal Q}^{\rm p},B}$ are pro-${\mathcal C}$ topologies on ${\mathbb Z}^n$.
\end{lemma}

\begin{proof} For each $k\in {\mathbb N}$ there holds $g_k\in{\mathbb Z}[t]$ and the hyperbolicity of $A$ implies that $g_k(A)$ is invertible. Thus ${\mathcal Q}^{\rm p} \subset {\mathcal Q}^{\rm a}$. For each $k\in {\mathbb N}$, the similarity of $A$ and $B$ implies the invertibility of $g_k(B)$. By the Smith normal forms of $g(A)$ and $g(B)$, the subgroups $N_{k,A}$ and $N_{k,B}$ have finite index in ${\mathbb Z}^n$. The filtered from below property for ${\mathcal N}_{ {\mathcal Q}^{\rm p},A}$ follows by (Lemma 2.2 in \cite{BM12}) wherein we showed that for $k_1,k_2\in {\mathbb N}$, the choice of $k_3 = k_1k_2$ leads to $g_{k_3} = q_1g_{k_1}$ and $g_{k_3} = q_2g_{k_2}$ for $q_1,q_2\in {\mathbb Z}[t]$, so that for $m\in {\mathbb Z}^n g_{k_3}(A)$ there exists $z\in {\mathbb Z}^n$ such that $m = zg_{k_3}(A)$, whence $m = zq_1(A)g_{k_1}(A) \in {\mathbb Z}^n g_{k_1}(A)$ and $m = z q_2(A)g_{k_2}(A) \in {\mathbb Z}^n g_{k_2}(A)$, implying that $N_{k_3,A} \subset N_{k_1,A}\cap N_{k_2,A}$. A similar argument gives $N_{k_3,B} \subset N_{k_1,B}\cap N_{k_2,B}$.
\end{proof}

For similar hyperbolic $A,B\in {\rm GL}_n({\mathbb Z})$, we show that ${\mathcal N}_{ {\mathcal Q}^{\rm p},A}$ and ${\mathcal N}_{ {\mathcal Q}^{\rm p},B}$ are each cofinal in ${\mathcal F}$, that ${\mathcal N}_{ {\mathcal Q}^{\rm p},A}$ is cofinal in ${\mathcal N}_{ {\mathcal Q}^{\rm a},A}$ and  ${\mathcal N}_{ {\mathcal Q}^{\rm p},B}$ is cofinal in ${\mathcal N}_{ {\mathcal Q}^{\rm a},B}$. We only state the results for $A$, as those for $B$ are obtained by replacing $A$ with $B$.

\begin{lemma}\label{dksubset} Suppose $A\in {\rm GL}_n({\mathbb Z})$ is hyperbolic. For every $d\in {\mathbb N}$ there exists $k\in{\mathbb N}$ such that $N_{k,A}\subset d{\mathbb Z}^n$.
\end{lemma}

\begin{proof} For fixed but arbitrary $d\in{\mathbb N}$, consider the finite subset of ${\mathbb T}^n$ given by
\[ T_d = \{ x\in {\mathbb T}^n: dx = 0\}.\]
This finite set consists entirely of periodic points of the left action of the hyperbolic toral automorphism $A$ of ${\mathbb T}^n$. Let $k$ be the least common multiple of the periods of $x\in T_d$ under the left action of $A$. Then $(A^k-I)x = 0$ for all $x\in T_d$, so that $T_d\subset {\rm Per}_k(A)$. Each rational point of the form $(0,\dots,0,1/d,0,\dots,0)^{\rm T}$ belongs to $T_d$ and satisfies
\[ (A^k-I)(0,\dots,0,1/d,0,\dots,0)^{\rm T} = 0\]
because $T_d\subset {\rm Per}_k(A)$. This implies that each entry of $A^k-I$ is divisible by $d$. Thus ${\mathbb Z}^n(A^k-I)\subset d{\mathbb Z}^n$, i.e., $N_{k,A}\subset d{\mathbb Z}^n$.
\end{proof}

\begin{proposition}\label{cofinal} Suppose $A\in {\rm GL}_n({\mathbb Z})$ is hyperbolic. Then ${\mathcal N}_{ {\mathcal Q}^{\rm p},A}$ is cofinal in ${\mathcal N}_{ {\mathcal Q}^{\rm a},A}$, and ${\mathcal N}_{ {\mathcal Q}^{\rm p},A}$ is cofinal in ${\mathcal F}$.
\end{proposition}

\begin{proof} By Lemma \ref{Qprin} there holds ${\mathcal Q}^{\rm p}\subset {\mathcal Q}^{\rm a}$ which implies ${\mathcal N}_{{\mathcal Q}^{\rm p},A}\subset {\mathcal N}_{ {\mathcal Q}^{\rm a},A}$. Furthermore, by Lemma \ref{Qprin} there holds ${\mathcal N}_{{\mathcal Q}^{\rm p},A} \subset {\mathcal F}$. For $N\in {\mathcal N}_{ {\mathcal Q}^{\rm a},A}$ there holds $N\leq_f {\mathbb Z}^n$ by Lemma \ref{Qall}. For $N\in {\mathcal F}$ there holds $N\leq_f {\mathbb Z}^n$. In either case we have $N\leq_f {\mathbb Z}^n$. By Lemma \ref{kproduct} there is a $d\in{\mathbb N}$ such that $d{\mathbb Z}^n\subset N$. Associated to this $d\in {\mathbb N}$ there is by Lemma \ref{dksubset} a $k\in{\mathbb N}$ such that $N_{k,A}\subset d{\mathbb Z}^n$. Hence $N_{k,A} \subset d{\mathbb Z}^n \subset N$, i.e., $N\preceq N_{k,A}$.
\end{proof}

\begin{corollary}\label{GAisoZnbar} If $A\in {\rm GL}_n({\mathbb Z})$ is hyperbolic, then
\[ {\mathcal K}_{ {\mathcal N}_{ {\mathcal Q}^{\rm p},A} }({\mathbb Z}^n) =  \varprojlim_{ k\in {\mathbb N}} {\mathbb Z}^n/N_{k,A}\]
is topologically isomorphic to ${\mathcal K}_{\mathcal F}({\mathbb Z}^n)$ and to
\[ {\mathcal K}_{ {\mathcal N}_{ {\mathcal Q}^{\rm a},A}} ( {\mathbb Z}^n) = \varprojlim_{g\in {\mathcal Q}^{\rm a}} {\mathbb Z}^n/ N_{g,A}.\]
\end{corollary}

\begin{proof} Apply Proposition \ref{cofinal} and Proposition \ref{topisoiota}.
\end{proof}

By Corollary \ref{GAisoZnbar}, applied to similar hyperbolic $A,B\in {\rm GL}_n({\mathbb Z})$, there are topological isomorphisms 
\[ {\mathcal K}_{ {\mathcal N}_{ {\mathcal Q}^{\rm a},A}}({\mathbb Z}^n) \cong {\mathcal K}_{ {\mathcal N}_ { {\mathcal Q}^{\rm p},A}}({\mathbb Z}^n) \cong {\mathcal K}_{\mathcal F}({\mathbb Z}^n) \cong {\mathcal K}_{ {\mathcal N}_{ {\mathcal Q}^{\rm p}, B}}({\mathbb Z}^n) \cong {\mathcal K}_{ {\mathcal N}_{ {\mathcal Q}^{\rm a},B} }({\mathbb Z}^n).\]
Thus, for similar hyperbolic $A,B \in {\rm GL}_n({\mathbb Z})$, the pro-${\mathcal C}$ completions ${\mathcal K}_{ {\mathcal N}_{ {\mathcal Q}^{\rm a},A}}({\mathbb Z}^n)$, ${\mathcal K}_{ {\mathcal N}_{ {\mathcal Q}^{\rm p},A}}({\mathbb Z}^n)$, ${\mathcal K}_{ {\mathcal N}_{ {\mathcal Q}^{\rm p},B}}({\mathbb Z}^n)$, and ${\mathcal K}_{ {\mathcal N}_{ {\mathcal Q}^{\rm a},B}}({\mathbb Z}^n)$, are indistinguishable as topological abelian groups. We will show that what does distinguish ${\mathcal K}_{ {\mathcal N}_{ {\mathcal Q}^{\rm a},A}}({\mathbb Z}^n)$ and ${\mathcal K}_{ {\mathcal N}_{{\mathcal Q}^{\rm p},A}}({\mathbb Z}^n)$ from ${\mathcal K}_{ {\mathcal N}_{ {\mathcal Q}^{\rm p},B}}({\mathbb Z}^n)$ and ${\mathcal K}_{ {\mathcal N}_{ {\mathcal Q}^ {\rm a},B}}({\mathbb Z}^n)$ is an $R$-module structure on these pro-${\mathcal C}$ completions.

\subsection{Module Structure}\label{sec:Module}

For similar hyperbolic $A,B \in {\rm GL}_n({\mathbb Z})$, there is a commutative ring $R$ with identity to which both matrix rings ${\mathbb Z}[A]$ and ${\mathbb Z}[B]$ are isomorphic. When the common characteristic polynomial $f(t)$ of $A$ and $B$ is irreducible the ring $R$ is the integral domain given by quotient
\[ R = {\mathbb Z}[t] / \big(f(t)\big)\]
where $\big( f(t) \big) = f(t){\mathbb Z}[t]$ is the principal ideal generated by $f(t)$. The map $\vartheta_A: {\mathbb Z}[t]/\big( f(t)\big) \to {\mathbb Z}[A]$ given by $q(t) + \big(f(t)\big) \to q(A)$, where
\[ {\mathbb Z}[A] = \left\{ \sum_{i=0}^{n-1} c_i A^i : c_i\in {\mathbb Z}\right\},\]
is a ring isomorphism, and the similarly defined map $\vartheta_B:{\mathbb Z}[t]/\big(f(t)\big) \to {\mathbb Z}[B]$ is also a ring isomorphism. When $f(t)$ is irreducible, this gives ${\mathbb Z}[A] \cong R \cong {\mathbb Z}[B]$ as rings. When $f(t)$ is reducible, the map $\vartheta_A$ is a ring epimorphism, and so by the First Isomorphism Theorem for Rings it induces a ring isomorphism
\[ \bar{\vartheta}_A: \big({\mathbb Z}[t]/\big( f(t)\big)\big) / {\rm ker}(\vartheta_A) \to {\mathbb Z}[A].\]
Similarly $\vartheta_B$ is a ring epimorphism, and so it induces a ring isomorphism
\[ \bar{\vartheta}_B: \big({\mathbb Z}[t]/\big( f(t)\big)\big) / {\rm ker}\vartheta_B \to {\mathbb Z}[B].\]
Similarity of $A$ and $B$ implies that ${\rm ker}(\vartheta_A) = {\rm ker}(\vartheta_B)$. Thus it follows that
\[ {\mathbb Z}[A] \cong \big( {\mathbb Z}[t]/\big(f(t)\big)/{\rm ker}(\vartheta_A) = \big( {\mathbb Z}[t]/\big(f(t)\big)/{\rm ker}(\vartheta_B) \cong {\mathbb Z}[B].\]
When $f(t)$ is reducible, the ring $R$ to which ${\mathbb Z}[A]$ and ${\mathbb Z}[B]$ are both isomorphic is $\big( {\mathbb Z}[t]/\big(f(t)\big)/{\rm ker}(\vartheta_A)$, which is an integral domain when the minimal polynomial of the similar $A$ and $B$ is irreducible, and is not an integral domain when the minimal polynomial of the similar $A$ and $B$ is reducible. With ${\mathcal Q} = {\mathcal Q}^{\rm p}$ or ${\mathcal Q}^{\rm a}$, when we refer to the $R$-module structure it is understood that it is ${\mathbb Z}[A]$ on ${\mathcal N}_{ {\mathcal Q},A}$, ${\mathcal K}_{ {\mathcal N}_{{\mathcal Q}, A}}({\mathbb Z}^n)$, etc., and it is ${\mathbb Z}[B]$ on ${\mathcal N}_{ {\mathcal Q},B}$, ${\mathcal K}_{ {{\mathcal N}_{\mathcal Q},B}}({\mathbb Z}^n)$, etc.

For a hyperbolic $A\in {\rm GL}_n({\mathbb Z})$ and ${\mathcal Q}$ equal to either ${\mathcal Q}^{\rm p}$ or ${\mathcal Q}^{\rm a}$, we describe the $R$-module structure on the pro-${\mathcal C}$ topology ${\mathcal N}_{{\mathcal Q},A} = \{ N_{g,A} : g\in {\mathcal Q}\}$, the pro-${\mathcal C}$ completion ${\mathcal K}_{ {\mathcal N}_ {{\mathcal Q},A}}( {\mathbb Z}^n)$, and the several natural homomorphisms associated to the pro-${\mathcal C}$ completion. (The $R$-module structure for hyperbolic $B$ similar to $A$ is described in like manner.)  The ring $R\cong{\mathbb Z}[A]$ acts on ${\mathbb Z}^n$ via the automorphism $m\to mA$ for $m\in{\mathbb Z}^n$, making ${\mathbb Z}^n$ into a right ${\mathbb Z}[A]$-module. Since each $N_{g,A} \in {\mathcal N}_{ {\mathcal Q},A}$ is invariant under the right action of $A$, the subgroup $N_{g,A}$ is a right ${\mathbb Z}[A]$-module. Denoting the elements
\[ G_{g,A} = {\mathbb Z}^n/N_{g,A}\]
by $[m]_{g,A} = m + N_{g,A}$ for $m\in {\mathbb Z}^n$, the right ${\mathbb Z}[A]$-module structure on $G_{g,A}$ is induced by the topological automorphism $A_g$ of $G_{g,A}$ defined by
\[ A_g( [m]_{g,A}) = [mA]_{g,A}.\]
For $h\preceq g$ the transition maps are the continuous epimorphisms $\varphi_{g,h,A}:G_{g,A}\to G_{h,A}$ defined by
\[ \varphi_{g,h,A}([m]_{g,A}) = [m]_{h,A}\]
which are ${\mathbb Z}[A]$-module epimorphisms. The profinite completion of ${\mathbb Z}^n$ with respect to the pro-${\mathcal C}$ topology determined by ${\mathcal N}_{{\mathcal Q},A}$ we denote by
\[ G_{ {\mathcal Q},A} = \varprojlim_{g\in{\mathcal Q}} G_{g,A}\]
instead of ${\mathcal K}_{\mathcal N_{ {\mathcal Q},A}}({\mathbb Z}^n)$ for simplicity of notation.
The elements of $G_{ {\mathcal Q},A}$ consist of those coherent tuples $\{ [m_g]_{g,A}\}_{g\in{\mathcal Q}}$ lying in $\prod_{g\in{\mathcal Q}} G_{g,A}$ (where $m_g\in {\mathbb Z}^n$ for each $g\in {\mathcal Q}$), i.e., those tuples that satisfy
\[ [m_h]_{h,A} = \varphi_{g,h,A}( [m_g]_{g,A}) = [m_g]_{h,A}\]
whenever $h\preceq g$. The right ${\mathbb Z}[A]$-module structure on $G_{ {\mathcal Q},A}$ is induced by the topological automorphism $\Gamma_A$ of $G_{ {\mathcal Q},A}$ defined by
\[ \Gamma_A ( \{ [m_g]_{g,A}\}_{g\in{\mathcal Q}}) = \{ A_g ([m_g]_{g,A})\}_{g\in {\mathcal Q}}.\]
For each $g\in{\mathcal Q}$, the continuous epimorphism $\varphi_{g,A}:G_{ {\mathcal Q},A}\to G_{g,A}$ defined by
\[ \varphi_{g,A}( \{ [m_h]_{h,A}\}_{h\in{\mathcal Q}}) = [m_g]_{g,A}\]
is a ${\mathbb Z}[A]$-module epimorphism. By $R$-module isomorphism theory, each epimorphism $\varphi_{g,A}$ induces the canonical $R$-module isomorphism
\[ G_{{\mathcal Q},A}/{\rm ker} \varphi_{g,A} \cong_R G_{g,A}.\]
The continuous homomorphism $\iota_{ {\mathcal Q},A}:{\mathbb Z}^n \to G_{ {\mathcal Q},A}$ (we use $\iota_{{\mathcal Q},A}$ instead of $\iota_{ {\mathcal N}_{ {\mathcal Q},A}}$ for simplicity of notation) defined by
\[ \iota_{ {\mathcal Q},A}(m) = \{ [m]_{g,A}\}_{g\in{\mathcal Q}}\]
is also a ${\mathbb Z}[A]$-module homomorphism.  The map $\iota_{ {\mathcal Q},A}$ is a monomorphism for both ${\mathcal Q} = {\mathcal Q}^{\rm p}$ and ${\mathcal Q} = {\mathcal Q}^{\rm a}$ by Corollary \ref{iotainjective} because ${\mathcal N}_{ {\mathcal Q}^{\rm p},A}$ is cofinal in ${\mathcal F}$ and ${\mathcal N}_{ {\mathcal Q}^{\rm p},A}$ is cofinal in ${\mathcal N}_{ {\mathcal Q}^{\rm a},A}$ by Proposition \ref{cofinal}, and $\iota_{\mathcal F}$ is a monomorphism by Proposition \ref{residual}. We proved directly that hyperbolicity of $A$ implies the intersection of $N_{k,A}$, $k\in {\mathbb N}$, is $\{0\}$ (see Lemma 2.2 in \cite{BM12}), hence, since ${\mathcal N}_{ {\mathcal Q}^{\rm p},A} \subset {\mathcal N}_{ {\mathcal Q}^{\rm a},A}$ by Proposition \ref{cofinal}, that
\[ \bigcap_{k\in{\mathbb N} } N_{k,A} = \{0\} \ \Rightarrow\ \bigcap_{ g\in {\mathcal Q}^{\rm a}} N_{g,A} = \{0\},\]
also showing that $\iota_{ {\mathcal Q},A}$ is a monomorphism for both ${\mathcal Q} = {\mathcal Q}^{\rm p}$ and ${\mathcal Q} = {\mathcal Q}^{\rm a}$.

\subsection{Profinite Conjugacy}

\begin{definition} For ${\mathcal Q}$ being either ${\mathcal Q}^{\rm p}$ or ${\mathcal Q}^{\rm a}$, two similar hyperbolic $A,B\in {\rm GL}_n({\mathbb Z})$ are profinitely conjugate over ${\mathcal Q}$ if there exists a topological $R$-module isomorphism $\Psi_{\mathcal Q}:G_{ {\mathcal Q},A}\to G_{ {\mathcal Q},B}$.
\end{definition}

Profinite conjugacy over ${\mathcal Q}$ for similar hyperbolic toral automorphisms is an equivalence relation, as is readily shown. 

\begin{remark}\label{sufficient} We showed that strong {\rm BF}-equivalence over ${\mathcal Q}^{\rm p}$ implies that $G_{ {\mathcal Q}^{\rm p},A}$ and $G_{ {\mathcal Q}^{\rm p},B}$ are topologically $R$-module isomorphic (see Theorem 4.1 in \cite{BM12}) where we used the $R$-module version of Theorem 3.27 in \cite{RZ10} (up to $R$-module isomorphism, $G_{ {\mathcal Q}^{\rm p},A}$ and $G_{ {\mathcal Q}^{\rm p},B}$ have the same finite $R$-module quotients, namely ${\rm BF}_g(A) \cong_R {\rm BF}_g(B)$ for all $g\in {\mathcal Q}^{\rm p}$). The same argument shows that strong {\rm BF}-equivalence over ${\mathcal Q}^{\rm a}$ implies that $G_{ {\mathcal Q}^{\rm a},A}$ and $G_{ {\mathcal Q}^{\rm a},B}$ are topologically $R$-module isomorphic. Thus, for ${\mathcal Q} = {\mathcal Q}^{\rm p}$ or ${\mathcal Q}^{\rm a}$, a sufficient condition for when two similar hyperbolic $A,B\in {\rm GL}_n({\mathbb Z})$ are profinitely conjugate over ${\mathcal Q}$, and hence in the same equivalence class, is strong {\rm BF}-equivalence of $A$ and $B$ over ${\mathcal Q}$, i.e., ${\rm BF}_g(A) \cong_R {\rm BF}_g(B)$ for all $g\in {\mathcal Q}$. 
\end{remark}

\begin{lemma}\label{QminQ} For a hyperbolic $A\in {\rm GL}_n({\mathbb Z})$ the profinite $R$-modules $G_{ {\mathcal Q}^{\rm p},A}$ and $G_{{\mathcal Q}^{\rm a},A}$ are topologically $R$-module isomorphic.
\end{lemma}

\begin{proof} By Corollary \ref{GAisoZnbar} the profinite $R$-modules $G_{ {\mathcal Q}^{\rm p},A}$ and $G_{ {\mathcal Q}^{\rm a},A}$ are topologically isomorphic as topological abelian groups. The proof of Corollary \ref{GAisoZnbar} uses Proposition \ref{topisoiota} where an explicit topological isomorphism is given. It is straightforward to verify that this explicitly given topological isomorphism is an $R$-module map, making $G_{ {\mathcal Q}^{\rm p},A}$ and $G_{ {\mathcal Q}^{\rm a},A}$ topologically $R$-module isomorphic.
\end{proof}

For ${\mathcal Q}$ being either ${\mathcal Q}^{\rm p}$ or ${\mathcal Q}^{\rm a}$, the pro-${\mathcal C}$ completions $G_{{\mathcal Q},A}$ and $G_{ {\mathcal Q},B}$ have, of course, the closest connection with the generalized Bowen-Franks $R$-modules ${\rm BF}_g(A)$ and ${\rm BF}_g(B)$ for $g\in {\mathcal Q}$. This is because
\[ {\rm BF}_g(A) = {\mathbb Z}^n/{\mathbb Z}^n g(A) = G_{g,A} \cong_R G_A/{\rm ker}\varphi_{g,A},\ g\in {\mathcal Q},\]
with a similar statement holding with $A$ replaced by $B$. The pro-${\mathcal C}$ completions $G_{{\mathcal Q},A}$ and $G_{ {\mathcal Q},B}$ are both topologically isomorphic to the full pro-${\mathcal C}$ completion ${\mathcal K}_{\mathcal F}({\mathbb Z}^n)$ by Corollary \ref{GAisoZnbar}, but not every finite index subgroup of ${\mathbb Z}^n$ that is used in the construction of ${\mathcal K}_{\mathcal F}({\mathbb Z}^n)$ is invariant by the corresponding right action of $A$ or $B$. Hence we use the dynamically defined $R$-modules $G_{{\mathcal Q},A}$ and $G_{ {\mathcal Q},B}$ to prove that profinite conjugacy over ${\mathcal Q}$ of similar hyperbolic $A$ and $B$ implies strong {\rm BF}-equivalence of $A$ and $B$ over ${\mathcal Q}$. We make use of the one-to-one correspondence in Theorem \ref{fullproCtopology}, via Remark \ref{applyfullproC}, from $\{ U: U\leq_o {\mathbb Z}^n; {\mathcal N}_{ {\mathcal Q},A}\}$ to $\{ V: V\leq_o G_{ {\mathcal Q},A}\}$ and its properties without explicit reference in what follows.

\subsection{The Converse} The main result is proven in this subsection. The idea of the proof of the main result is to mimic the quotient of the generalized Bowen-Franks $R$-modules in the profinite $R$-module setting and connect it with generalized Bowen-Franks $R$-modules through the isomorphism theory for $R$-modules. The first step in this direction is to show through a series of four lemmas that, for a pro-${\mathcal C}$ topology ${\mathcal N}_{ {\mathcal Q},A}$ on ${\mathbb Z}^n$ with ${\mathcal Q}$ being either ${\mathcal Q}^{\rm p}$ or ${\mathcal Q}^{\rm a}$, the closure of the image $g(\Gamma_A)(G_{ {\mathcal Q},A})$ is equal to ${\rm ker}\varphi_{g,A}$ for each $g\in {\mathcal Q}$.

\begin{lemma}\label{intersection} For a hyperbolic $A\in {\rm GL}_n({\mathbb Z})$ and $g\in{\mathcal Q}$, there holds
\[\iota_{ {\mathcal Q},A}({\mathbb Z}^n g(A))\subset g(\Gamma_A) (G_{ {\mathcal Q},A})\cap \iota_{ {\mathcal Q},A}({\mathbb Z}^n).\]
\end{lemma}

\begin{proof} Suppose $x \in \iota_{ {\mathcal Q},A}( {\mathbb Z}^n g(A))$. Then $x=\iota_{ {\mathcal Q},A}(m g(A))$ for some $m\in{\mathbb Z}^n$. Hence $x = g(\Gamma_A)(\iota_{ {\mathcal Q},A}(m)) \in g(\Gamma_A)(G_{ {\mathcal Q},A})$. Since $mg(A)\in{\mathbb Z}^n$, we have $x \in \iota_{ {\mathcal Q},A}({\mathbb Z}^n)$, and so
\[ x \in g(\Gamma_A)(G_{ {\mathcal Q},A}) \cap \iota_{ {\mathcal Q},A}({\mathbb Z}^n).\]
This implies that $\iota_{ {\mathcal Q},A}({\mathbb Z}^n g(A) )\subset g(\Gamma_A)(G_{ {\mathcal Q},A}) \cap \iota_{ {\mathcal Q},A}({\mathbb Z}^n)$.
\end{proof}

\begin{lemma}\label{contains} For a hyperbolic $A\in {\rm GL}_n({\mathbb Z})$ and $g\in{\mathcal Q}$, there holds
\[ g(\Gamma_A)(G_{ {\mathcal Q},A}) \subset  {\rm ker}\varphi_{g,A}.\] 
\end{lemma}

\begin{proof} For $\{ [m_h]_{h,A}\}_{h\in{\mathcal Q}}\in G_{ {\mathcal Q},A}$, we have
\[ g(\Gamma_A)(\{ [m_h]_{h,A}\}_{h\in{\mathcal Q}}) = \{ [m_h g(A)]_{h,A}\}_{h\in{\mathcal Q}}.\]
Then
\[ \varphi_{g,A}( g(\Gamma_A) (\{ [m_h]_{h,A}\}_{h\in{\mathcal Q}})) = [m_g g(A)]_{g,A} = 0\]
because $m_g g(A)\in {\mathbb Z}^n g(A)$. Thus $g(\Gamma_A)(G_{ {\mathcal Q},A})  \subset  {\rm ker}\varphi_{g,A}$.
\end{proof}

\begin{lemma}\label{kernel} For a hyperbolic $A\in {\rm GL}_n({\mathbb Z})$ and $g\in{\mathcal Q}$, there holds
\[ {\rm ker}\varphi_{g,A} \cap \iota_{ {\mathcal Q},A}({\mathbb Z}^n) = \iota_{ {\mathcal Q},A}( {\mathbb Z}^n g(A)).\]
\end{lemma}

\begin{proof} Suppose $x\in \iota_{ {\mathcal Q},A}({\mathbb Z}^n g(A))$. Then $x=\iota_{ {\mathcal Q},A}(m g(A))$ for some $m\in{\mathbb Z}^n$. Hence $\iota_{ {\mathcal Q},A}(m g(A))\in \iota_{ {\mathcal Q},A}({\mathbb Z}^n)$, and 
\[ \varphi_{g,A} (x) = [mg(A)]_{g,A} = 0\]
because $m g(A)\in {\mathbb Z}^n g(A)$. Thus $\iota_{ {\mathcal Q},A}({\mathbb Z}^n g(A))\subset {\rm ker}\varphi_{g,A}\cap \iota_{ {\mathcal Q},A}({\mathbb Z}^n)$.

Now suppose that $x \in {\rm ker}\varphi_{g,A} \cap \iota_{ {\mathcal Q},A}({\mathbb Z}^n)$. Then $x=\iota_{ {\mathcal Q},A}(m)$ for some $m\in{\mathbb Z}^n$ and
\[ 0 = \varphi_{g,A} (\iota_{ {\mathcal Q},A}(m)) =  [m]_{g,A}.\]
This says that $m\in {\mathbb Z}^n g(A)$, so that $x = \iota_{ {\mathcal Q},A}(m) \in \iota_{ {\mathcal Q},A}({\mathbb Z}^n g(A))$. This implies that ${\rm ker}\varphi_{g,A}\cap \iota_{ {\mathcal Q},A}({\mathbb Z}^n) \subset \iota_{ {\mathcal Q},A}({\mathbb Z}^n g(A))$.

Therefore ${\rm ker}\varphi_{g,A} \cap \iota_{ {\mathcal Q},A}({\mathbb Z}^n) = \iota_{ {\mathcal Q},A}( {\mathbb Z}^n g(A))$.
\end{proof}

\begin{lemma}\label{open} For a hyperbolic $A\in {\rm GL}_n({\mathbb Z})$ and $g\in{\mathcal Q}$, there holds
\[ \overline{ g(\Gamma_A)(G_{ {\mathcal Q},A}) } = {\rm ker}\varphi_{g,A}.\]
\end{lemma}

\begin{proof} By Lemma \ref{contains}, we have $\overline{ g(\Gamma_A)(G_{ {\mathcal Q},A}) } \subset {\rm ker}\varphi_{g,A}$ because ${\rm ker}\varphi_{g,A}$ is closed.

On the other hand, by Lemma \ref{kernel}, we have ${\rm ker}\varphi_{g,A} \cap \iota_{ {\mathcal Q},A}({\mathbb Z}^n) = \iota_{ {\mathcal Q},A}( {\mathbb Z}^n g(A))$. Since ${\rm ker}\varphi_{g,A} \leq_o G_{ {\mathcal Q},A}$, we have that
\[ \overline{\iota_{ {\mathcal Q},A}({\mathbb Z}^n g(A))} = \overline{ {\rm ker}\varphi_{g,A}\cap \iota_{ {\mathcal Q},A}({\mathbb Z}^n)}= {\rm ker}\varphi_{g,A}.\]
By Lemma \ref{intersection}, we have $\iota_{ {\mathcal Q},A}({\mathbb Z}^n g(A)) \subset g(\Gamma_A)(G_{ {\mathcal Q},A})\cap \iota_{ {\mathcal Q},A}({\mathbb Z}^n)$, so that
\begin{align*} \overline{\iota_{ {\mathcal Q},A}({\mathbb Z}^n g(A) )} 
& \subset \overline{ g(\Gamma_A)(G_{ {\mathcal Q},A}) \cap \iota_{ {\mathcal Q},A}({\mathbb Z}^n) } \\
& \subset \overline{ g(\Gamma_A)(G_{ {\mathcal Q},A}) } \cap \overline{\iota_{ {\mathcal Q},A}({\mathbb Z}^n)} \\
& = \overline{ g(\Gamma_A)(G_{ {\mathcal Q},A}) } \cap G_{ {\mathcal Q},A} \\
& = \overline{ g(\Gamma_A)(G_{ {\mathcal Q},A})}.
\end{align*}
This gives ${\rm ker}\varphi_{g,A} \subset \overline{ g(\Gamma_A)(G_{ {\mathcal Q},A})}$.

Therefore $\overline{ g(\Gamma_A)(G_{ {\mathcal Q},A})} = {\rm ker}\varphi_{g,A}$.
\end{proof}

For any profinite conjugacy $\Psi_{\mathcal Q}:G_{ {\mathcal Q},A}\to G_{ {\mathcal Q},B}$ over ${\mathcal Q}$ there holds 
\[  g(\Gamma_B)\Psi_{\mathcal Q} =   \Psi_{\mathcal Q}g(\Gamma_A) {\rm\ for\ all\ }g\in{\mathcal Q},\]
because $\Psi_{\mathcal Q}$ is an $R$-module isomorphism, i.e., $ \Gamma_B\Psi_{\mathcal Q} =  \Psi_{\mathcal Q}\Gamma_A $.

\begin{lemma}\label{congkpower} For similar hyperbolic $A,B\in {\rm GL}_n({\mathbb Z})$, if $\Psi_{\mathcal Q}:G_{{\mathcal Q},A}\to G_{ {\mathcal Q},B}$ is a profinite conjugacy over ${\mathcal Q}$, then for all $g\in{\mathcal Q}$ there holds
\[  \Psi_{\mathcal Q}\big( g(\Gamma_A)(G_{ {\mathcal Q},A})  \big) = g(\Gamma_B)(G_{ {\mathcal Q},B}) .\]
\end{lemma}

\begin{proof} Let $\Psi_{\mathcal Q}:G_{ {\mathcal Q},A}\to G_{ {\mathcal Q},B}$ be a topological $R$-module isomorphism. 

Suppose $x\in g(\Gamma_A)(G_{ {\mathcal Q},A})$. Then there exists $y\in G_{ {\mathcal Q},A}$ such that $x = g(\Gamma_A)(y)$. Set $z= \Psi_{\mathcal Q}(y) \in G_{ {\mathcal Q},B}$. Then we have
\[ \Psi_{\mathcal Q}(x) = \Psi_{\mathcal Q} g(\Gamma_A) (y) =  g(\Gamma_B)\Psi_{\mathcal Q} (y) =  g(\Gamma_B)(z). \]
Thus $\Psi_{\mathcal Q}(x) \in g(\Gamma_B)(G_{ {\mathcal Q},B}) $, and so $\Psi_{\mathcal Q}( g(\Gamma_A)(G_{ {\mathcal Q},A}) ) \subset  g(\Gamma_B) (G_{ {\mathcal Q},B}) $.

Now let $x\in g(\Gamma_B)(G_{ {\mathcal Q},B})$. Then there is $y\in G_{ {\mathcal Q},B}$ such that $x =  g(\Gamma_B)(y)$. Because $\Psi_{\mathcal Q}$ is a bijection, there exists a unique  $z\in G_{ {\mathcal Q},A}$ such that $\Psi_{\mathcal Q}(z) = y$. Set $w = g(\Gamma_A)(z)$. Then we have
\[ \Psi_{\mathcal Q} (w)= \Psi_{\mathcal Q}\big(g(\Gamma_A)(z) \big) =  g(\Gamma_B)\Psi_{\mathcal Q} (z) =  g(\Gamma_B)(y) = x.\]
This gives $x\in \Psi_{\mathcal Q}( g(\Gamma_A)(G_{ {\mathcal Q},A})  )$, and so $g(\Gamma_B)( G_{ {\mathcal Q},B}) \subset \Psi_{\mathcal Q} ( g(\Gamma_A)(G_{ {\mathcal Q},A})  )$.

Therefore $\Psi_{\mathcal Q}( g(\Gamma_A)(G_{ {\mathcal Q},A})  ) = g(\Gamma_B)(G_{ {\mathcal Q},B}) $.
\end{proof}

\begin{theorem}\label{main} Suppose that $A,B\in {\rm GL}_n({\mathbb Z})$ are similar and hyperbolic. If $A$ and $B$ are profinitely conjugate over ${\mathcal Q}$, then $A$ and $B$ are strongly {\rm BF}-equivalent over ${\mathcal Q}$.
\end{theorem}

\begin{proof} Suppose $\Psi_{\mathcal Q}:G_{ {\mathcal Q},A}\to G_{ {\mathcal Q},B}$ is a profinite conjugacy over ${\mathcal Q}$. By Lemma \ref{congkpower}, we have for each $g\in{\mathcal Q}$ that $\Psi_{\mathcal Q}( g(\Gamma_A)(G_{ {\mathcal Q},A})) = g(\Gamma_B)(G_{ {\mathcal Q},B})$. Since $\Psi_{\mathcal Q}$ is a homeomorphism, we have for each $g\in {\mathcal Q}$ that
\[ \Psi_{\mathcal Q}\left(\overline{g(\Gamma_A)(G_{ {\mathcal Q},A})}\right)  = \overline{ g(\Gamma_B)(G_{ {\mathcal Q},B}) }.\]
By Lemma \ref{open} for each $g\in {\mathcal Q}$ there holds $\overline{ g(\Gamma_A)(G_{ {\mathcal Q},A}) } = {\rm ker}\varphi_{g,A}$ and $\overline{ g(\Gamma_B)(G_{ {\mathcal Q},B}) } = {\rm ker}\varphi_{g,B}$, and so
\[ \Psi_{\mathcal Q}( {\rm ker}\varphi_{g,A})  = {\rm ker}\varphi_{g,B}.\]
Thus, by the isomorphism theory for $R$-modules, $\Psi_{\mathcal Q}$ induces an $R$-module isomorphism $G_{ {\mathcal Q},A}/{\rm ker}\varphi_{g,A} \cong_R G_{ {\mathcal Q},B}/{\rm ker}\varphi_{g,B}$ for each $g\in {\mathcal Q}$. Since $G_{g,A} \cong_R G_{ {\mathcal Q},A}/{\rm ker}\varphi_{g,A}$ and $G_{g,B}\cong_R G_{ {\mathcal Q},B}/{\rm ker}\varphi_{g,B}$, we obtain $G_{g,A} \cong_R G_{g,B}$ for all $g\in{\mathcal Q}$. Thus $A$ and $B$ are strongly {\rm BF}-equivalent over ${\mathcal Q}$.
\end{proof}

\begin{corollary}\label{characterization} Suppose $A,B\in {\rm GL}_n({\mathbb Z})$ are similar and hyperbolic. Then $A$ and $B$ are profinitely conjugate over ${\mathcal Q}$ if and only if $A$ and $B$ are strongly {\rm BF}-equivalent over ${\mathcal Q}$.
\end{corollary}

\begin{proof} Combine Theorem \ref{main} and Remark \ref{sufficient}.
\end{proof}

\begin{corollary}\label{prinall} Suppose $A,B\in {\rm GL}_n({\mathbb Z})$ are similar and hyperbolic. The following are equivalent.
\begin{enumerate}
\item[(a)] $A$ and $B$ are strongly {\rm BF}-equivalent over ${\mathcal Q}^{\rm p}$.
\item[(b)] $A$ and $B$ are profinitely conjugate over ${\mathcal Q}^{\rm p}$.
\item[(c)] $A$ and $B$ are profinitely conjugate over ${\mathcal Q}^{\rm a}$.
\item[(d)] $A$ and $B$ are strongly {\rm BF}-equivalent over ${\mathcal Q}^{\rm a}$.
\end{enumerate}
\end{corollary}

\begin{proof} Parts (a) and (b) are equivalent by Corollary \ref{characterization} applied when ${\mathcal Q} = {\mathcal Q}^{\rm p}$ and parts (c) and (d) are equivalent by Corollary \ref{characterization} applied when ${\mathcal Q} = {\mathcal Q}^{\rm a}$. By Lemma \ref{QminQ} the profinite $R$-modules $G_{ {\mathcal Q}^{\rm p},A}$ and $G_{ {\mathcal Q}^{\rm a},A}$ are topologically $R$-module isomorphic and the profinite $R$-modules $G_{ {\mathcal Q}^{\rm p},B}$ and $G_{ {\mathcal Q}^{\rm a},B}$ are topologically $R$-module isomorphic. Hence $G_{ {\mathcal Q}^{\rm p},A}$ and $G_{ {\mathcal Q}^{\rm p},B}$ are profinite conjugate over ${\mathcal Q}^{\rm p}$ if and only if $G_{ {\mathcal Q}^{\rm a},A}$ and $G_{ {\mathcal Q}^{\rm a},B}$ are profinitely conjugate over ${\mathcal Q}^{\rm a}$, which is the equivalence of parts (b) and (c).
\end{proof}

The implications of Corollary \ref{prinall} are that profinite conjugacy over ${\mathcal Q}^{\rm p}$ and over ${\mathcal Q}^{\rm a}$ are equivalent and that the principal generalized Bowen-Franks $R$-modules are indeed the principal invariants among the generalized Bowen-Franks $R$-modules parameterized by ${\mathcal Q}^{\rm a}$. Going forward we need only be concerned with profinite conjugacy of similar hyperbolic $A,B\in {\rm GL}_n({\mathbb Z})$ over ${\mathcal Q}^p$. To aid with this, we simplify the notation by writing ${\mathcal N}_A = {\mathcal N}_{ {\mathcal Q}^{\rm p},A}$, $G_A = G_{ {\mathcal Q}^{\rm p},A}$, $\iota_A = \iota_{ {\mathcal Q}^{\rm P},A}$, an element of $G_A$ as $\{ [m_k]_{k,A}\}_{k\in {\mathbb N}}$, ${\mathcal N}_B = {\mathcal N}_{ {\mathcal Q}^{\rm p},B}$ $G_B = G_{ {\mathcal Q}^{\rm p},B}$, $\iota_B = \iota_{ {\mathcal Q}^{\rm p},B}$, and an element of $G_B$ as $\{ [m_k]_{k,B}\}_{ k\in {\mathbb N}}$. We also refer to a profinite conjugacy over ${\mathcal Q}^{\rm p}$ between $A$ and $B$ simply as a profinite conjugacy $\Psi:G_A\to G_B$ between $A$ and $B$, and to strong {\rm BF}-equivalence of $A$ and $B$ over ${\mathcal Q}^{\rm p}$ simply as strong {\rm BF}-equivalence of $A$ and $B$.

\section{Properties of Profinite Conjugacies} 

\subsection{Two Characterizations of Conjugacy}

The existence of a profinite conjugacy between two similar hyperbolic toral automorphisms is insufficient to imply their conjugacy, as it is known that strong ${\rm BF}$-equivalence is insufficient to imply conjugacy (see \cite{BM19, MR05}). However, for similar hyperbolic toral automorphisms $A,B\in {\rm GL}_n({\mathbb Z})$, there are two characterizations of when profinite conjugacies between $A$ and $B$ gives conjugacy of $A$ and $B$ in terms of the embedded copies $\iota_A({\mathbb Z}^n)$ and $\iota_B({\mathbb Z}^n)$ of ${\mathbb Z}^n$, as we showed in (Theorem 5.1 in \cite{BM12}):
\begin{enumerate}
\item[(a)] $A$ and $B$ are conjugate, 
\item[(b)] there is a profinite conjugacy $\Psi:G_A\to G_B$ such that $\Psi(\iota_A({\mathbb Z}^n)) = \iota_B({\mathbb Z}^n)$, and 
\item[(c)] there is a profinite conjugacy $\Psi:G_A\to G_B$ such that $\Psi(\iota_A({\mathbb Z}^n))\cap \iota_B({\mathbb Z}^n)$ is $R$-module isomorphic to the right ${\mathbb Z}[B]$-module ${\mathbb Z}^n$.
\end{enumerate}
We structured the proof of this in the cyclic format (a) $\Rightarrow$ (b), (b) $\Rightarrow$ (c), and (c) $\Rightarrow$ (a). The proof that (a) $\Leftrightarrow$ (b) follows from the proof of (a) $\Rightarrow$ (b)  and an extra independent proof of (b) $\Rightarrow$ (a) included in \cite{BM12}, the latter highlighting the pro-${\mathcal C}$ topology ${\mathcal N}_B$ on ${\mathbb Z}^n$. But there was an error in the proof of (c) $\Rightarrow$ (a) in \cite{BM12}, which, using the notation
\[ \Delta_\Psi = \Psi( \iota_A({\mathbb Z}^n)) \cap \iota_B({\mathbb Z}^n),\]
leads to following correction of part (c):
\begin{enumerate}
\item[(c)] there is a profinite conjugacy $\Psi:G_A\to G_B$ such that $\Psi^{-1}( \Delta_\Psi)$ is $R$-module isomorphic to the right ${\mathbb Z}[A]$-module ${\mathbb Z}^n$ and $\Delta_\Psi$ is $R$-module isomorphic to the right ${\mathbb Z}[B]$-module ${\mathbb Z}^n$.
\end{enumerate}
With the inclusion of the additional condition of the correction in part (c), here are the corrected parts of the proof of the two characterization of the conjugacy.

(b) $\Rightarrow$ (c). Suppose there is a profinite conjugacy $\Psi:G_A\to G_B$ such that $\Psi(\iota_A({\mathbb Z}^n)) = \iota_B({\mathbb Z}^n)$. Then $\Delta_\Psi= \iota_B({\mathbb Z}^n)$. The map $\iota_B:{\mathbb Z}^n\to G_B$ from the right ${\mathbb Z}[B]$-module to the right ${\mathbb Z}[B]$-module $G_B$ is an $R$-module monomorphism. So the map $\iota_B^{-1}: \iota_B({\mathbb Z}^n)\to {\mathbb Z}^n$ is an $R$-module isomorphism from $\iota_B({\mathbb Z}^n) = \Delta_\Psi$ to the right ${\mathbb Z}[B]$-module ${\mathbb Z}^n$. On the other hand, from $\Delta_\Psi = \Psi(\iota_A({\mathbb Z}^n))$ we have $ \Psi^{-1}( \Delta_\Psi) = \iota_A({\mathbb Z}^n)$. The map $\iota_A:{\mathbb Z}^n\to G_A$ is an $R$-module monomorphism from the right ${\mathbb Z}[A]$-module ${\mathbb Z}^n$ to the right ${\mathbb Z}[A]$-module $G_A$. So the map $\iota_A^{-1}: \iota_A({\mathbb Z}^n)\to {\mathbb Z}^n$ is an $R$-module isomorphism from the right ${\mathbb Z}[A]$-module $\iota_A({\mathbb Z}^n) = \Psi^{-1}(\Delta_\Psi)$ to the right ${\mathbb Z}[A]$-module ${\mathbb Z}^n$. This gives (c).

(c) $\Rightarrow$ (a). The assumptions of part (c) imply the existence of $R$-module isomorphisms $h:{\mathbb Z}^n\to \Delta_\Psi$ and $g:{\mathbb Z}^n \to \Psi^{-1}(\Delta_\Psi)$ that satisfy
\[ B h^{-1} = h^{-1} \Gamma_B,\ \ gA = \Gamma_A g.\]
[The additional condition in the corrected version of part (c) gives the existence of $g$, while the original condition in part (c) gives the existence of $h$.] The mapping $h^{-1}\Psi g:{\mathbb Z}^n\to {\mathbb Z}^n$ is an $R$-module isomorphism from the right ${\mathbb Z}[A]$-module ${\mathbb Z}^n$ to the right ${\mathbb Z}[B]$-module ${\mathbb Z}^n$. [The original incorrect proof of (c) $\Rightarrow$ (a) used the $R$-module mapping $h^{-1}\Psi \iota_A$ which is where the error occurred.] This implies that $h^{-1}\Psi g$ is an automorphism of ${\mathbb Z}^n$, and gives the existence of $D\in {\rm GL}_n({\mathbb Z})$ such that $D = h^{-1}\Psi g$. This is illustrated in the following diagram.
\[ \begin{CD}
\Psi^{-1}(\Delta_\Psi) @>\Psi>> \Delta_\Psi  \\ 
@AgAA            @ AAhA \\
{\mathbb Z}^n  @>>D>   {\mathbb Z}^n
\end{CD}
\]
Using the right $R$-module structures of $h^{-1}$, $\Psi$, and $g$, there holds
\[ DA = h^{-1}\Psi gA = h^{-1}\Psi \Gamma_A g = h^{-1}\Gamma_B \Psi g = Bh^{-1} \Psi g = BD.\]
Setting $C = D^{-1}$ gives $AC = CB$, hence (a).

We state the corrected version of the two characterizations of conjugacy for similar hyperbolic toral automorphisms.

\begin{theorem}\label{twochar}
For similar hyperbolic $A,B\in {\rm GL}_n({\mathbb Z})$, the following are equivalent.
\begin{enumerate}
\item[(a)] $A$ and $B$ are conjugate in ${\rm GL}_n({\mathbb Z})$, i.e., there is $C\in {\rm GL}_n({\mathbb Z})$ such that $AC=CB$,
\item[(b)] there is a profinite conjugacy $\Psi:G_A\to G_B$ such that $\Psi(\iota_A({\mathbb Z}^n)) = \iota_B({\mathbb Z}^n)$, and 
\item[(c)] there is a profinite conjugacy $\Psi:G_A\to G_B$ such that $\Psi^{-1}( \Delta_\Psi)$ is $R$-module isomorphic to the right ${\mathbb Z}[A]$-module ${\mathbb Z}^n$ and $\Delta_\Psi$ is $R$-module isomorphic to the right ${\mathbb Z}[B]$-module ${\mathbb Z}^n$.
\end{enumerate}
\end{theorem}

\subsection{The Embedded Copy of ${\mathbb Z}^n$}

The two characterizations of conjugacy in Theorem \ref{twochar} involve the embedded dense copies of ${\mathbb Z}^n$ through the $R$-module monomorphisms $\iota_A$ and $\iota_B$. The second characterization of conjugacy in part (c) of Theorem \ref{twochar} has the presence of the $R$-submodule $\Delta_\Psi=\Psi(\iota_A({\mathbb Z}^n))\cap\iota_B({\mathbb Z}^n)$. When similar hyperbolic $A$ and $B$ are conjugate, there exists a profinite conjugacy by Theorem \ref{twochar} for which $\Psi(\iota_A({\mathbb Z}^n)) = \iota_B({\mathbb Z}^n)$ so that $\Delta_\Psi = \iota_B({\mathbb Z}^n)$. In general, however, part (c) of Theorem \ref{twochar} poses the question of if it is possible for a profinite conjugacy $\Psi$ to send $\iota_A({\mathbb Z}^n)$ into a proper $R$-submodule of $\iota_B({\mathbb Z}^n)$. We answer this in the negative.

\begin{lemma}\label{properRsubmodule} For similar hyperbolic $A,B\in {\rm GL}_n({\mathbb Z})$, if $\Psi:G_A\to G_B$ is a profinite conjugacy between $A$ and $B$, then $\Psi(\iota_A({\mathbb Z}^n))$ is not a proper $R$-submodule of $\iota_B({\mathbb Z}^n)$.
\end{lemma}

\begin{proof} Suppose $\Psi:G_A\to G_B$ is profinite conjugacy for which $\Psi(\iota_A({\mathbb Z}^n))$ is a proper $R$-submodule of $\iota_B({\mathbb Z}^n)$. Since $\iota_A({\mathbb Z}^n)\leq_o {\mathbb Z}^n$ (with respect to the pro-${\mathcal C}$ topology on ${\mathbb Z}^n$ determined by ${\mathcal N}_A$ and the identification of $\iota_A({\mathbb Z}^n)$ with ${\mathbb Z}^n$), we have $\overline{\iota_A({\mathbb Z}^n)} \leq_o G_A$. Then $\Psi\left( \overline{\iota_A({\mathbb Z}^n)}\right) \leq_o G_B$, and so by the identification of ${\mathbb Z}^n$ with $\iota_B({\mathbb Z}^n)$ we have
\[ \Psi\left( \overline{\iota_A({\mathbb Z}^n)}\right) \cap \iota_B({\mathbb Z}^n)\leq_o {\mathbb Z}^n\]
(with respect to the pro-${\mathcal C}$ topology on ${\mathbb Z}^n$ determined by ${\mathcal N}_B$). Let $H\leq_o {\mathbb Z}^n$ be the set defined by
\[ \iota_B(H) = \Psi\left( \overline{\iota_A({\mathbb Z}^n)}\right) \cap \iota_B({\mathbb Z}^n).\]
Since $\Psi$ is a homeomorphism we have
\[ \iota_B (H)  = \overline{ \Psi(\iota_A({\mathbb Z}^n))} \cap \iota_B({\mathbb Z}^n).\]
By hypothesis, $\Psi(\iota_A({\mathbb Z}^n))$ is a proper $R$-submodule of $\iota_B({\mathbb Z}^n)$. By the $1-1$ correspondence from $\{ U: U\leq_o {\mathbb Z}^n; {\mathcal N}_B\}$ to $\{ V:V\leq_o G_B\}$ there is only one open subgroup of ${\mathbb Z}^n$ (with respect to pro-${\mathcal C}$ topology determined by ${\mathcal N}_B$) whose closure is $G_B$, and this open subgroup is ${\mathbb Z}^n$. Since $\Psi(\iota_A({\mathbb Z}^n))\ne \iota_B({\mathbb Z}^n)$ it follows that $H$ is a proper open subgroup of ${\mathbb Z}^n$, for if $H={\mathbb Z}^n$, then
\[ \overline{ \Psi(\iota_A({\mathbb Z}^n)) }\cap \iota_B({\mathbb Z}^n) = \iota_B({\mathbb Z}^n)\]
which would imply that $\overline{\Psi(\iota_A({\mathbb Z}^n))} = G_B$, giving $\Psi(\iota_A({\mathbb Z}^n)) = \iota_B({\mathbb Z}^n)$, a contradiction. Hence $[{\mathbb Z}^n:H]>1$. Since $\overline{{\mathbb Z}^n}=G_B$, we have
\[ [G_B:\overline{H}] = [{\mathbb Z}^n:H]>1.\]
But since $\iota_A({\mathbb Z}^n)$ is dense in $G_A$ and since $\Psi$ is a homeomorphism, the image $\Psi(\iota_A({\mathbb Z}^n))$ is dense in $G_B$. Hence
\[ \overline{\Psi(\iota_A({\mathbb Z}^n))} = G_B,\]
which implies that $H = {\mathbb Z}^n$, so that $[G_B:\overline{H}] = 1$, a contradiction.
\end{proof}

\subsection{Cyclic Automorphisms}

For profinitely conjugate similar hyperbolic toral automorphisms $A$ and $B$ we show how $\iota_A({\mathbb Z}^n)$ and $\iota_B({\mathbb Z}^n)$ within $G_A$ and $G_B$ distinguish $A$ from $B$ when $A$ is cyclic and $B$ is not cyclic. An automorphism $A\in {\rm GL}_n({\mathbb Z})$ is called cyclic if there exists $\xi\in {\mathbb Z}^n$ such that the set $\{ \xi, \xi A,\dots, \xi A^{n-1}\}$ forms a basis for ${\mathbb Z}^n$, and $\xi$ is called a generator for $A$, i.e., the right ${\mathbb Z}[A]$-module ${\mathbb Z}^n$ is cyclic, or equivalently, $\iota_A({\mathbb Z}^n)$ is a cyclic $R$-submodule of $G_A$. As is well known, a cyclic automorphism $A$ is conjugate to the companion matrix of its characteristic polynomial $f(t) = a_0 + a_1t + \cdots + a_{n-1} t^{n-1} + t^n \in {\mathbb Z}[t]$, i.e., to
\[ C_f=\begin{bmatrix} 
0         &        1 &         0 &  \dots  & 0         & 0      \\
0         &        0 &         1 &  \dots  & 0         & 0      \\
0         &        0 &         0 &  \dots  & 0          & 0      \\
\vdots & \vdots & \vdots & \ddots & \vdots  & \vdots     \\
0         &        0 &        0  &  \dots  & 0          & 1 \\
-a_0    &   -a_1 &  -a_2  &  \dots  &-a_{n-2} & -a_{n-1} 
\end{bmatrix}.\]
In particular, cyclicity is an invariant of the conjugacy class of the companion matrix. Profinite conjugacy realizes this invariant for hyperbolic toral automorphisms in a very explicit manner.

\begin{theorem}\label{cyclicity} Suppose that $A,B\in {\rm GL}_n({\mathbb Z})$ are similar and hyperbolic. If $A$ is cyclic with generator $\xi\in{\mathbb Z}^n$, and $B$ is not cyclic, then for any profinite conjugacy $\Psi:G_A\to G_B$ there holds 
\[ \Psi(\iota_A(\xi))\not\in \iota_B({\mathbb Z}^n).\]
\end{theorem}

\begin{proof} Suppose $A$ is cyclic with generator $\xi\in{\mathbb Z}^n$. Then the set $\{\xi,\xi A,\dots,\xi A^{n-1}\}$ is a basis for ${\mathbb Z}^n$. Identifying $\xi$ with its image $\iota_A(\xi)$ in $G_A$, the set
\[ \big\{ \{ [\xi A^j]_{k,A}\}_{k\in {\mathbb N}}: j=0,1,\dots,n-1\big\}\]
is a basis for $\iota_A({\mathbb Z}^n)$. For a topological $R$-module isomorphism $\Psi:G_A\to G_B$,
set
\[ \{ [\eta_k]_{k,B}\}_{k\in{\mathbb N}} = \Psi( \{ [\xi]_{k,A}\}_{k\in{\mathbb N}}.\]
Since $\Psi$ is a $R$-module isomorphism, we have
\begin{align*}
\Psi( \{ [\xi A]_{k,A}\}_{k\in{\mathbb N}}
& = \Psi ( \{ A_k[\xi]_{k,A}\}_{k\in{\mathbb N}} \\
& = (\Psi\circ \Gamma_A)(\{ [\xi]_{k,A}\}_{k\in{\mathbb N}} \\
& = (\Gamma_B\circ \Psi) (\{ [\xi]_{k,A}\}_{k\in{\mathbb N}} \\
& = \Gamma_B ( \{ [\eta_k]_{k,B}\}_{k\in{\mathbb N}} \\
& = \{ B_k[\eta_k]_{k,B}\}_{k\in{\mathbb N}} \\
& = \{ [\eta_k B]_{k,B}\}_{k\in{\mathbb N}}.
\end{align*}
By induction we have for $1\leq r\leq n-1$ that
\[ \Psi(\{ [\xi A^r]_{k,A}\}_{k\in {\mathbb N}}) = \{ [\eta_k B^r]_{k,B}\}_{k\in{\mathbb N}}.\]

Suppose to the contrary that $\Psi(\iota_A(\xi))\in \iota_B({\mathbb Z}^n)$. Then $\{ [\eta_k]_{k,B}\}_{k\in{\mathbb N}}\in \iota_B({\mathbb Z}^n)$. Thus there is $m^\prime\in {\mathbb Z}^n$ such that $[\eta_k]_{k,B} = [m^\prime]_{k,B}$ for all $k\in{\mathbb N}$. For each $1\leq r\leq n-1$ we have
\[ \{ [\eta_k B^r]_{k,B}\}_{k\in{\mathbb N}} = \{ B_k^r[\eta_k]_{k,B}\}_{k\in{\mathbb N}} = 
\{ B_k^r[m^\prime]_{k,B}\}_{k\in{\mathbb N}} = \{ [m^\prime B^r]_{k,B}\}_{k\in{\mathbb N}}.
\]
This implies that $\{ [n_k B^r]_{k,B}\}_{k\in{\mathbb N}}\in \iota_B({\mathbb Z}^n)$ for all $1\leq r\leq n-1$.

We show that $\Psi(\iota_A({\mathbb Z}^n)) \subset \iota_B({\mathbb Z}^n)$. Let $m\in {\mathbb Z}^n$. Since $\{ \xi, \xi A,\dots, \xi A^{n-1}\}$ is a basis for ${\mathbb Z}^n$ there exist $c_1,c_2,\dots,c_n\in{\mathbb Z}$ such that
\[ m = c_1\xi + c_2 \xi A + \cdots + c_n \xi A^{n-1}.\]
Then
\begin{align*}
\Psi(\iota_A(m)) 
& = \Psi ( \{ [m]_{k,A}\}_{k\in{\mathbb N}}) \\
& = \Psi ( \{ [c_1\xi + c_2\xi A + \cdots + c_n \xi A^{n-1}]_{k,A}\}_{k\in{\mathbb N}}) \\
& = \Psi ( \{ c_1[\xi]_{k,A} + c_2[\xi A]_{k,A} + \cdots + c_n [\xi A^{n-1}]_{k,A}\}_{k\in{\mathbb N}}) \\
& = \Psi ( c_1 \{ [\xi]_{k,A}\}_{k\in{\mathbb N}} + c_2 \{ [\xi A]_{k,A}\}_{k\in{\mathbb N}} + \cdots + c_n \{ [\xi A^{n-1}]_{k,A}\}_{k\in{\mathbb N}} ) \\
& = c_1\{ [m^\prime]_{k,B}\}_{k\in{\mathbb N}} + c_2\{ [m^\prime B]_{k,B}\}_{k\in{\mathbb N}} + \cdots +  c_n\{ [m^\prime B^{n-1}]_{k,B}\}_{k\in{\mathbb N}} \\
& \in \iota_B({\mathbb Z}^n).
\end{align*}
Since $m\in{\mathbb Z}^n$ is arbitrary, we obtain $\Psi(\iota_A({\mathbb Z}^n)) \subset \iota_B({\mathbb Z}^n)$.

We show that $\Psi(\iota_A({\mathbb Z}^n)) \ne \iota_B({\mathbb Z}^n)$.  Using $m^\prime$ from above, if every element of $\iota_B({\mathbb Z}^n)$ could be written in the form
\[ c_1\{ [m^\prime]_{k,B}\}_{k\in{\mathbb N}} + c_2\{ [m^\prime B]_{k,B}\}_{k\in{\mathbb N}} + \cdots +  c_n\{ [m^\prime B^{n-1}]_{k,B}\}_{k\in{\mathbb N}}\]
for $c_i\in {\mathbb Z}$, $i=1,2,\dots,n$, then $B$ would be cyclic with generator $m^\prime$. But $B$ is not cyclic by hypothesis, and so $\Psi(\iota_A({\mathbb Z}^n)) \ne \iota_B({\mathbb Z}^n)$.

We now have $\Psi(\iota_A({\mathbb Z}^n))$ is a proper $R$-submodule of $\iota_B({\mathbb Z}^n)$. This contradicts Lemma \ref{properRsubmodule}. Therefore, $\Psi(\iota_A(\xi))\not\in \iota_B({\mathbb Z}^n)$.
\end{proof}

\subsection{Distinguishing Conjugacy Classes}

Profinite conjugacy distinguishes between the conjugacy classes of similar hyperbolic toral automorphisms according to the contrapositive of Theorem \ref{twochar} and the assumption of strong {\rm BF}-equivalence which by Corollary \ref{characterization} gives the existence of profinite conjugacies.  For similar, hyperbolic, strongly {\rm BF}-equivalent $A,B\in {\rm GL}_n({\mathbb Z})$ the following are equivalent:
\begin{enumerate}
\item[(a)] $A$ and $B$ are not conjugate,
\item[(b)] {\it for all} profinite conjugacies $\Psi:G_A\to G_B$ there holds $\Psi(\iota_A({\mathbb Z}^n)) \ne \iota_B({\mathbb Z}^n)$, and
\item[(c)] {\it for all} profinite conjugacies  $\Psi:G_A\to G_B$ either $\Psi^{-1}(\Delta_\Psi)$ is not $R$-module isomorphic to right ${\mathbb Z}[A]$-module ${\mathbb Z}^n$ or $\Delta_\Psi$ is not $R$-module isomorphic to the right ${\mathbb Z}[B]$-module ${\mathbb Z}^n$ (where $\Delta_\Psi = \Psi(\iota_A({\mathbb Z}^n))\cap \iota_B({\mathbb Z}^n)$).
\end{enumerate}

Parts (a), (b), and (c) of this equivalence have particular realizations for similar, hyperbolic, strongly {\rm BF}-equivalent $A,B\in {\rm GL}_n({\mathbb Z})$ when $A$ is cyclic, with generator $\xi$, and $B$ is not cyclic. The realization of (a) follows because cyclicity is an invariant of conjugacy. The realization of (b) follows because {\it for all} profinite conjugacies $\Psi:G_A\to G_B$ there holds $\Psi(\iota_A(\xi)) \not\in \iota_B({\mathbb Z}^n)$ by Theorem \ref{cyclicity}, which implies {\it for all} profinite conjugacies $\Psi:G_A\to G_B$ that $\Psi(\iota_A({\mathbb Z}^n))\ne \iota_B({\mathbb Z}^n)$. The realization of (c) follows because {\it for all} profinite conjugacies $\Psi:G_A\to G_B$ the $R$-submodules $\Psi^{-1}(\Delta_\Psi)$ and $\Delta_\Psi$ are $R$-module isomorphic, while the right ${\mathbb Z}[A]$-module ${\mathbb Z}^n$ is cyclic and the right ${\mathbb Z}[B]$-module ${\mathbb Z}^n$ is not cyclic, implying {\it for all} profinite conjugacies $\Psi:G_A\to G_B$ that either $\Psi^{-1}(\Delta_\Psi)$ is not $R$-module isomorphic to the right ${\mathbb Z}[A]$-module ${\mathbb Z}^n$ or $\Delta_\Psi$ is not $R$-module isomorphic to the right ${\mathbb Z}[B]$-module ${\mathbb Z}^n$.


\begin{thebibliography}{99}

\bibitem{BF77}
Bowen~R, Franks~J. Homology for Zero-Dimensional Nonwandering Sets, Ann of Math, 106 (1977) 73-92.

\bibitem{LM21}
Lind~D, Marcus~B. An Introduction to Symbolic Dynamics and Coding, 2nd ed, Cambridge University Press, New York, 2021.

\bibitem{PS75}
Parry~B, Sullivan~D. A Topological Invariant of Flows on $1$-Dimensional Spaces, Topology, 14 (1975) 297-299.

\bibitem{Fr84}
Franks~J. Flow equivalence of subshifts of finite type, Ergod Th Dynam Sys, 4 (1984), 53-66.

\bibitem{Hu95}
Huang~D. Flow equivalence of reducible shifts of finite type and Cuntz-Krieger algebras, J reine angew Math, 462 (1995) 185-217.

\bibitem{Ma01}
Matsumoto~ K. Bowen-Franks groups as an invariant for flow equivalence of subshifts, Ergod Th Dynam Sys, 21 (2001) 1831-1842.

\bibitem{Ma21}
Matsumoto~K. Flow equivalence of topological Markov shifts and Ruelle algebras, New York J Math, 27 (2021) 1375-1414. 

\bibitem{Ki98}
Kitchens~B. Symbolic Dynamics: One-sided, Two-sided and Countable State Markov Shifts, Springer Universitext, Berlin, 1998.

\bibitem{MR05}
Martins Rodrigues~P, Sousa Ramos~ J. Bowen-Franks groups as conjugacy invariants for ${\mathbb T}^n$-automorphisms, Aequationes Math, 69 (2005) 231-249.

\bibitem{AP65}
Adler~RL, Palais~ R. Homeomorphic Conjugacy of Automorphisms of the Torus, Proc Amer Math Soc, 16 (1965) 1222-1225.

\bibitem{Ch09}
Chen~S. Generalized Bowen-Franks groups of integral matrices with the same zeta function, Linear Algebra and its Applications, 431 (2009) 1397-1406.

\bibitem{BM12}
Bakker~LF, Martins Rodrigues~P. A profinite group invariant for hyperbolic toral automorphisms, Discrete Contin Dyn Syst, 32 (2012) 1965-1976.

\bibitem{LM33}
Lattimer~C, MacDuffee~C. A correspondence between classes of ideals and classes of matrices, Ann Math, 34 (1933) 313-316.

\bibitem{Ta78}
Taussky~O. Connections between algebraic number theory and integral matrices, In: Cohn~H. A Classical Invitation to Algebraic Numbers and Class Fields, Springer-Verlag, New York,1978.

\bibitem{DTZ62}
Dade~EC, Taussky~O, Zassenhaus~ H. On the theory of orders, in particular on the semigroup of ideal classes and genera of an order in an algebraic number field, Math Ann, 148 (1962) 31-64.

\bibitem{BMM16}
Bakker~ LF, Martins Rodrigues~P, Moreira~MMR. Local conjugacy of irreducible hyperbolic toral automorphisms, Available at ArXiv:1611.05551.

\bibitem{BM19}
Bakker~LF, Martins Rodrigues~P. Block Conjugacy of Irreducible Toral Automorphisms, Dynamical Systems 34 (2019) 244-258.

\bibitem{RZ10}
Ribes~L, Zalesskii~P. Profinite Groups, 2nd ed, Springer, New York, 2010.

\bibitem{Bo89}
Bourbaki~N. General Topology, Chapters 1-4. Springer, New York, 1989.

\bibitem{Hu74}
Hungerford~ TW. Algebra, Springer, New York, 1974.

\end{thebibliography}
\end{document}